\documentclass[a4paper,10pt]{article}

\usepackage{amsmath, amsthm, amscd, amsfonts, amssymb, graphicx, color}

\newcommand*{\medcap}{\mathbin{\scalebox{1.5}{\ensuremath{\cap}}}}%

\usepackage{enumitem}
\usepackage{url}
\usepackage{multirow}
\usepackage{subfigure}
\usepackage[pass]{geometry}
\usepackage{cite}
\usepackage{cancel}

\usepackage{tikz}
\usetikzlibrary{calc,shapes,decorations.pathreplacing,matrix}

\tikzset{square matrix/.style={
    matrix of nodes,
    column sep=-\pgflinewidth, row sep=-\pgflinewidth,
    nodes={draw,
      minimum height=4.5pt,
      anchor=center,
      text width=4.5pt,
      align=center,
      inner sep=0pt
    },
  },
  square matrix/.default=1.2cm
}

\newtheorem{theo}{Theorem}
\newtheorem{lemma}{Lemma}

\newtheorem{conj}{Conjecture}
\newtheorem{prop}{Proposition}

\begin{document}

\title{Cartesian product graphs and $k$-tuple total domination}

\author{Adel P. Kazemi$^{1,3}$, Behnaz Pahlavsay$^{1,4}$ and Rebecca J.\ Stones$^{2,5,}$\footnote{Stones supported by her NSF China Research Fellowship for International Young Scientists (grant number: 11550110491).} \\[1em]
$^1$ Department of Mathematics, University of Mohaghegh Ardabili, \\ P.O.\ Box 5619911367, Ardabil, Iran. \\
$^2$ College of Computer and Control Engineering, \\ Nankai University, China. \\[1em]
$^3$ Email: \url{adelpkazemi@yahoo.com} \\
$^4$ Email: \url{pahlavsayb@yahoo.com}, \url{pahlavsay@uma.ac.ir} \\
$^5$ Email: \url{rebecca.stones82@gmail.com} \\[1em]
}

\maketitle

\begin{abstract}
A $k$-tuple total dominating set ($k$TDS) of a graph $G$ is a set $S$ of vertices in which every vertex in $G$ is adjacent to at least $k$ vertices in $S$; the minimum size of a $k$TDS is denoted $\gamma_{\times k,t}(G)$.  We give a Vizing-like inequality for Cartesian product graphs, namely $\gamma_{\times k,t}(G) \gamma_{\times k,t}(H) \leq 2k \gamma_{\times k,t}(G \Box H)$ provided $\gamma_{\times k,t}(G) \leq 2k\rho(G)$, where $\rho$ is the packing number.  We also give bounds on $\gamma_{\times k,t}(G \Box H)$ in terms of (open) packing numbers, and consider the extremal case of $\gamma_{\times k,t}(K_n \Box K_m)$, i.e., the rook's graph, giving a constructive proof of a general formula for $\gamma_{\times 2, t}(K_n \Box K_m)$.
\\[0.2em]

\noindent
Keywords: $k$-tuple total domination, Cartesian product of graphs, rook's graph, Vizing's conjecture.
\\[0.2em]

\noindent
MSC(2010): 05C69.
\end{abstract}

\section{Introduction}

Domination is well-studied in graph theory and the literature on this subject has been surveyed and detailed in the two books by Haynes, Hedetniemi, and Slater~\cite{HHS5, HHS6}. Among the many variations of domination, the ones relevant to this paper is $k$-tuple domination and $k$-tuple total domination, which were introduced by Harary and Haynes \cite{HH3}, and by Henning and Kazemi \cite{HK8}, respectively.  Throughout this paper, we use standard notation as listed in Table~\ref{ta:notat}.  All graphs considered here are finite, undirected, and simple.

\begin{table}[htp]
\centering
\begin{tabular}{|p{1.62in}|p{3.55in}|}
\hline
\renewcommand{\arraystretch}{2}
$G=(V,E)$ & A graph with \emph{vertex set} $V=V(G)$ and \emph{edge set} $E=E(G)$. \\
$N_G(v)=\{u\in V: uv \in E\}$ & The \emph{open neighborhood} of vertex $v$ in $G$. \\
$N_G[v]=N_G(v) \cup \{v\}$ & The \emph{closed neighborhood} of vertex $v$ in $G$. \\
$\mathrm{deg}_G(v)=|N_{G}(v)|$ & The \emph{degree} of a vertex $v$ in $G$. \\
$\delta(G)$, $\Delta(G)$ & The \emph{minimum degree} and \emph{maximum degree} of vertices in $G$. \\
$C_n$, $K_n$ & The $n$-vertex \emph{cycle} and \emph{complete graph}. \\
\hline
$G \Box H$ & The Cartesian product of graphs $G$ and $H$. \\
$K_n \Box K_m$ & The $n \times m$ rook's graph. \\
\hline
$\gamma_{\times k}(G)$ & The $k$-tuple domination number of $G$. \\
$\gamma(G)=\gamma_{\times 1}(G)$ & The domination number of $G$. \\
$\gamma_{\times k,t}(G)$ & The $k$-tuple total domination number of $G$. \\
$\gamma_t(G)=\gamma_{\times 1,t}(G)$ & The total domination number of $G$. \\
\hline
$\rho(G)$ & The maximum cardinality of a packing (packing number). \\
$\rho^{(\mathrm{open})}(G)$ & The maximum cardinality of an open packing (open packing number). \\
\hline
\end{tabular}
\caption{Table of notation.}\label{ta:notat}
\end{table}

For a graph $G=(V,E)$ and $k \geq 1$, a set $S \subseteq V$ is called a $k$-\emph{tuple total dominating set} ($k$TDS) if every vertex $v \in V$ has at least $k$ neighbors in $S$, i.e., $|N_G(v)\cap S| \geq k$.  A $k$-\emph{tuple dominating set} ($k$DS) instead satisfies $|N_G[v]\cap S| \geq k$.  The \emph{$k$-tuple domination number} and the \emph{$k$-tuple total domination number}, which we denote $\gamma_{\times k}(G)$ and $\gamma_{\times k,t}(G)$, respectively, is the minimum cardinality of a $k$DS and a $k$TDS of $G$, respectively.  The familiar \emph{domination number} is thus $\gamma(G)=\gamma_{\times 1}(G)$.  We use min-$k$DS and min-$k$TDS to refer to $k$DSs and $k$TDSs of minimum size, respectively.

\begin{lemma}
For a graph to have a $k$-tuple dominating set (resp.\ $k$-tuple total dominating set), every vertex must have at least $k-1$ (resp.\ $k$) neighbors.
\end{lemma}

For example, for $k \geq 1$, a $k$-regular graph $G=(V,E)$ will have only one $k$-tuple total dominating set, namely $V$ itself.

The \emph{Cartesian product} $G \Box H$ of two graphs $G$ and $H$ is the graph with vertex set $V(G)\times V(H)$ where two vertices $(u_{1},v_{1})$ and $(u_{2},v_{2})$ are adjacent if and only if either $u_{1}=u_{2}$ and $v_{1}v_{2}\in E(H)$ or $v_{1}=v_{2}$ and $u_{1}u_{2}\in E(G)$.  For more information on product graphs see \cite{IK}.  We will be particularly interested in the case when $K_n \Box K_m$, which is known as the $n \times m$ \emph{rook's graph}, as edges represent possible moves by a rook on an $n \times m$ chess board.  The $3 \times 4$ rook's graph is drawn in Figure~\ref{fi:rook34}, along with a min-$2$TDS.

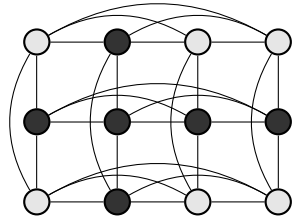
\begin{figure}[htp]
\centering
\begin{tikzpicture}
\matrix[nodes={draw, thick, fill=black!10, circle},row sep=0.7cm,column sep=0.7cm] {
  \node(11){}; &
  \node[fill=black!80](12){}; &
  \node(13){}; &
  \node(14){}; \\
  \node[fill=black!80](21){}; &
  \node[fill=black!80](22){}; &
  \node[fill=black!80](23){}; &
  \node[fill=black!80](24){}; \\
  \node(31){}; &
  \node[fill=black!80](32){}; &
  \node(33){}; &
  \node(34){}; \\
};
\draw (11) to (12) to (13) to (14); \draw (11) to[bend left] (13); \draw (12) to[bend left] (14); \draw (11) to[bend left] (14);
\draw (21) to (22) to (23) to (24); \draw (21) to[bend left] (23); \draw (22) to[bend left] (24); \draw (21) to[bend left] (24);
\draw (31) to (32) to (33) to (34); \draw (31) to[bend left] (33); \draw (32) to[bend left] (34); \draw (31) to[bend left] (34);
\draw (11) to (21) to (31) to[bend left] (11);  \draw (12) to (22) to (32) to[bend left] (12);  \draw (13) to (23) to (33) to[bend left] (13);  \draw (14) to (24) to (34) to[bend left] (14);
\end{tikzpicture}
\caption{The $3 \times 4$ rook's graph, i.e., $K_3 \Box K_4$.  The dark vertices highlight a min-$2$TDS, so $\gamma_{\times 2,t}(K_3 \Box K_4)=6$.}\label{fi:rook34}
\end{figure}

In 1963, and more formally in 1968, Vizing \cite{Viz} made an elegant conjecture that has subsequently become one the most famous open problems in domination theory.
\begin{conj}[Vizing's Conjecture]
For any graphs $G$ and $H$, $$\gamma(G)\gamma(H) \leq \gamma(G\Box H).$$
\end{conj}
Over more than forty years (see \cite{BDG} and references therein), Vizing's Conjecture is has been shown to hold for certain restricted classes of graphs, and furthermore, upper and lower bounds on the inequality have gradually tightened.  Additionally, research has explored inequalities (including Vizing-like inequalities) for different forms of domination \cite{HHS6}. A significant breakthrough occurred in 2000, when Clark and Suen \cite{CS1} proved that  $$\gamma (G)\gamma (H)\leq 2\gamma (G\Box H)$$ which led to the discovery of a Vizing-like inequality for total domination \cite{HR9,HPT}, i.e.,
\begin{equation}\label{eq:HRbound}
\gamma_t(G)\gamma_t(H) \leq 2 \gamma_t(G \Box H),
\end{equation}
as well as for paired \cite{BHRo,HJ,CM}, and fractional domination \cite{FRDM}, and the $\{k\}$-domination function (integer domination) \cite{BM,HL,CMH}, and total $\{k\}$-domination function \cite{HL}.

In this paper, we investigate inequalities for $k$-tuple total domination, i.e., we present lower and upper bounds on $\gamma_{\times k,t}(G \Box H)$ in terms of the orders of $G$ and $H$, the packing numbers and open packing numbers, and in terms of $\gamma_{\times k,t}(G)$ and $\gamma_{\times k,t}(H)$.  For example, Theorem~\ref{th:totalprodbound} gives a partial generalization of \eqref{eq:HRbound}.  We also find a formulas for $\gamma_{\times k,t}(K_n \Box K_m)$, and determine the value of $\gamma_{\times 2, t}(K_n \Box K_m)$ in Proposition~\ref{pr:k2count} for all $n$ and $m$.

Burchett, Lane, and Lachniet \cite{BLL} and Burchett \cite{B2011} found bounds and exact formulas for the $k$-tuple domination number and $k$-domination number of the rook's graph in square cases, i.e., $K_n \Box K_n$ (where $k$-domination is like total $k$-tuple total domination, but only vertices outside of the domination set need to be dominated).  The $k$-tuple total domination number is known for $K_n \times K_m$ \cite{HK7} and bounds are given for supergeneralized Petersen graphs \cite{KP}.

\section{General graphs}

A subset $S\subseteq V(G)$ is a \emph{packing} (resp.\ \emph{open packing}) if the closed (resp.\ open) neighborhoods of vertices in $S$ are pairwise disjoint. The \emph{packing number} (resp.\ \emph{open packing number}) of $G$, denoted $\rho(G)$ (resp.\ $\rho^{(\mathrm{open})}(G)$), is the maximum cardinality of a packing (resp.\ an open packing).  Note that vertices in packings $S$ have distance at least $3$, i.e., if $u,v\in S$, then $\mathrm{dist}_{G}(u,v)\geq 3$.  The following two lemmas are from \cite{HK7}.

\begin{lemma}
If $G$ is an $n$-vertex graph with $\delta(G) \geq k$, then $\gamma_{\times k,t}(G) \geq \lceil kn/\Delta(G) \rceil$.
\end{lemma}

\begin{proof}
The sum of the degrees of the vertices in any min-$k$TDS $D$ is at least $kn$ (since every vertex has at least $k$ neighbors in $D$) and at most $|D|\Delta(G)$ (by definition of maximum degree).  Hence $|D|\Delta(G) \geq kn$ and the lemma follows since $|D|=\gamma_{\times k,t}(G)$, by definition.
\end{proof}

\begin{lemma}\label{lm:tot-packing}
If $G$ is a graph with $\delta (G)\geq k$, then $\gamma_{\times k,t}(G) \geq k \rho^{(\mathrm{open})}(G) \geq k\rho(G)$.
\end{lemma}

\begin{proof}
A $k$TDS must contain $k$ vertices from each of the $\rho^{(\mathrm{open})}(G)$ disjoint open neighborhoods in any maximal open packing.  The second inequality is because every packing in a graph is also open packing.
\end{proof}

The following theorem gives an upper bound on the product of the $k$-tuple total domination numbers of two graphs in terms of the $k$-tuple total domination number of their Cartesian product.

\begin{theo}\label{th:totalprodbound}
Let $G$ and $H$ be two graphs, and suppose $\delta(H) \geq k$.  Then $$\rho(G) \gamma_{\times k,t}(H) \leq \gamma_{\times k,t}(G\Box H).$$  Hence, if $\delta(G) \geq k$ and $\gamma_{\times k,t}(G) \leq 2k\rho(G)$, then $$\gamma_{\times k,t}(G) \gamma_{\times k,t}(H) \leq 2k \gamma_{\times k,t}(G\Box H).$$
\end{theo}

\begin{proof}
Let $S$ be a min-$k$TDS of $G \Box H$.  Choose a maximal packing $P:=\{v_i\}_{i=1}^{\rho(G)}$ of $G$, and for each vertex $v_i$ in the packing, let $H_i$ be the subgraph  of $G \Box H$ induced by $\{v_i\} \times V(H)$.  An example is drawn in Figure~\ref{fi:packH1H2}.

\begin{figure}[htp]
\centering
\begin{tikzpicture}

\draw[dotted,fill=black!20] (2.55cm,0) ellipse (0.37 and 1.7);
\draw[dotted,fill=black!20] (5.6cm,0) ellipse (0.37 and 1.7);

\matrix[nodes={draw, thick, fill=black!10, circle},row sep=0.5cm,column sep=0.4cm] {
  & & & & & & & \node(h11){}; & &
  \node(11){}; &
  \node(12){}; &
  \node(13){}; &
  \node(14){}; &
  \node(15){}; &
  \node(16){}; \\
  
  \node(g11){}; &
  \node[star,star points=5,star point ratio=0.4](g12){}; &
  \node(g13){}; &
  \node(g14){}; &
  \node(g15){}; & 
  \node[star,star points=5,star point ratio=0.4](g16){}; & 
  \node[draw=none,fill=none] (box) {$\Box$}; & \node(h21){}; & \node[draw=none,fill=none] (box) {$=$}; &
  \node(21){}; &
  \node(22){}; &
  \node(23){}; &
  \node(24){}; &
  \node(25){}; &
  \node(26){}; \\
  
  & & & & & & & \node(h31){}; & &
  \node(31){}; &
  \node(32){}; &
  \node(33){}; &
  \node(34){}; &
  \node(35){}; &
  \node(36){}; \\
};
\draw (g11) to (g12) to (g13) to (g14) to (g15) to (g16); \draw (g11) to[bend left] (g13); \draw (g14) to[bend left] (g16);
\draw (11) to (12) to (13) to (14) to (15) to (16); \draw (11) to[bend left] (13); \draw (14) to[bend left] (16);
\draw (21) to (22) to (23) to (24) to (25) to (26); \draw (21) to[bend left] (23); \draw (24) to[bend left] (26);
\draw (31) to (32) to (33) to (34) to (35) to (36); \draw (31) to[bend left] (33); \draw (34) to[bend left] (36);
\draw (11) to (21) to (31) to[bend left] (11);  \draw (12) to (22) to (32) to[bend left] (12);  \draw (13) to (23) to (33) to[bend left] (13);  \draw (14) to (24) to (34) to[bend left] (14);  \draw (15) to (25) to (35) to[bend left] (15);  \draw (16) to (26) to (36) to[bend left] (16);
\draw (h11) to (h21) to (h31) to[bend left] (h11);

\node (h1) at ($(32) + (0,-0.9)$) {$H_1$};
\node (h1) at ($(36) + (0,-0.9)$) {$H_2$};
\end{tikzpicture}
\caption{A Cartesian product graph $G \Box H$.  The vertices of $G$ drawn as stars highlight a maximum packing $P$ of $G$, and are used to identify $H_1$ and $H_2$ in $G \Box H$.}\label{fi:packH1H2}
\end{figure}
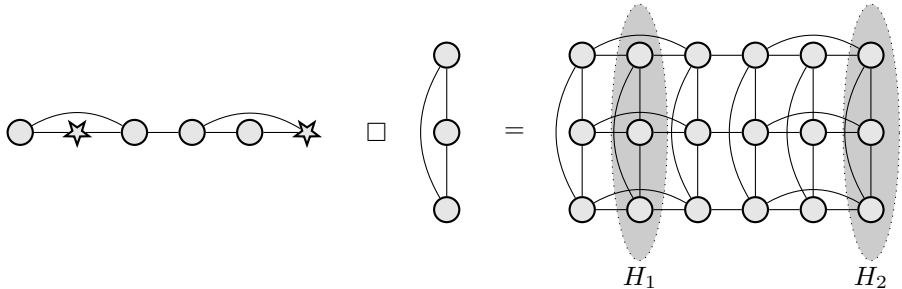

We partition $S$ into parts (a) $S_i$, for $i \in \{1,2,\ldots,\rho(G)\}$, containing the vertices of $S$ which are in or are adjacent to vertices in $H_i$, and (b) $X$, the remaining vertices (if any).  (The sets $S_i$ are disjoint, since $P$ is a packing of $G$.)  Hence
\begin{equation}\label{eq:Sbound1}
|S| \geq |S \setminus X | = \left| \bigcup_{i=1}^{\rho(G)} S_i \right|=\sum_{i=1}^{\rho(G)} |S_i|
\end{equation}
since the sets $S_i$ are disjoint.  Moreover, every vertex in $H_i$ has at least $k$ neighbors in $S_i$.

From $S_i$, we can construct a $k$TDS $D$ of $H_i$ of size at most $|S_i|$ as follows:
\begin{itemize}
 \item Add every vertex in $S_i \cap H_i$ to $D$.
 \item For each $x \in S_i \setminus H_i$, by definition of Cartesian product, $x$ has a unique neighbor in $H_i$; call it $x'$.
 \begin{itemize}
  \item If $x'$ has $k$ or more neighbors in $D$, do nothing.
  \item Otherwise, since $\delta(H) \geq k$, we know $x'$ has a neighbor $x''$ in $H_i \setminus D$.  Add $x''$ to $D$.
 \end{itemize}
\end{itemize}
Essentially, any $x \in S_i \setminus H_i$ dominates a unique vertex $x' \in H_i$ so, if necessary, we replace it by some unused $x'' \in H_i \cap N(x')$ which also dominates $x'$.  After performing these operations, $|S_i| \geq |D| \geq \gamma_{\times k,t}(H)$.  Thus
\begin{align*}
\gamma_{\times k,t}(G \Box H) & = |S| \\
& \geq\sum_{i=1}^{\rho(G)} |S_{i}| & [\text{by \eqref{eq:Sbound1}}] \\
& \geq \sum_{i=1}^{\rho(G)} \gamma_{\times k,t}(H) & [\text{by the above}] \\
& = \rho(G) \gamma_{\times k,t}(H). & \qedhere
\end{align*}
\end{proof}

The second part of Theorem~\ref{th:totalprodbound} is applicable when $\gamma_{\times k,t}(G) \leq 2k\rho(G)$; in contrast Lemma~\ref{lm:tot-packing} implies that $\gamma_{\times k,t}(G) \geq k\rho(G)$ holds when $\delta(G) \geq k$.

When $k=1$, i.e., total domination, Theorem~\ref{th:totalprodbound} gives the bound \eqref{eq:HRbound} when $\gamma_t(G) \leq 2\rho(G)$.  Equality holds in Theorem~\ref{th:totalprodbound} when $k=1$ in some instances:  Modifying a construction in \cite{HR9}, we take a graph $G=(V,E)$ and (a) add at least one pendant vertex to each vertex in $V$, then (b) subdivide each edge in $E$ twice.  Call the result $G^*$.  Then $V$ is both a maximum packing and a minimum dominating set of $G^*$.  So $\rho(G^*)=\gamma(G^*)=|V|=n$ and, in fact, we also find $\gamma_t (G^*)=2n$.  Figure~\ref{fi:HR9exam} illustrates an example of this construction.  Further, since $V(G) \times V(K_2)$ is a $(2n)$-vertex total dominating set of $G^* \Box K_2$, we have $$\gamma_t(G^*) \gamma_t(K_2) = 2 \gamma_t(G^* \Box K_2).$$

To further illustrate, the Petersen graph $\mathcal{P}$ has packing number $\rho(\mathcal{P})=1$ and we compute:
\begin{align*}
k=1: \qquad & \gamma_{\times 1,t}(\mathcal{P})=4>2k\rho(\mathcal{P}), \\
k=2: \qquad & \gamma_{\times 2,t}(\mathcal{P})=8>2k\rho(\mathcal{P}), \\
k=3: \qquad & \gamma_{\times 3,t}(\mathcal{P})=10>2k\rho(\mathcal{P}),
\end{align*}
so $\gamma_{\times k,t}(\mathcal{P}) \leq 2k\rho(\mathcal{P})$ is not satisfied in all three cases.  However, we can still apply the second part of Theorem~\ref{th:totalprodbound} when $H$ is the Petersen graph and $G$ is some other graph which satisfies $\delta(G) \geq k$ and $\gamma_{\times k,t}(G) \leq 2k\rho(G)$.

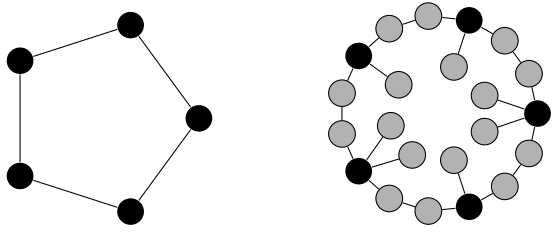
\begin{figure}[htp]
\centering
\begin{tikzpicture}[scale=1.3]
\tikzstyle{vertex}=[circle,fill=black,minimum size=10pt,inner sep=0pt]

\node[vertex] (v1) at (0*360/5:1) {};
\node[vertex] (v2) at (1*360/5:1) {};
\node[vertex] (v3) at (2*360/5:1) {};
\node[vertex] (v4) at (3*360/5:1) {};
\node[vertex] (v5) at (4*360/5:1) {};

\draw (v1) -- (v2) -- (v3) -- (v4) -- (v5) -- (v1);
\end{tikzpicture}
\qquad\qquad
\begin{tikzpicture}[scale=1.3]
\tikzstyle{vertex}=[circle,fill=black,minimum size=10pt,inner sep=0pt]
\tikzstyle{vertex2}=[circle,draw,fill=black!30,minimum size=10pt,inner sep=0pt]

\node[vertex] (v1) at (0*360/5:1) {};
\node[vertex2] (v12a) at (0.33333*360/5:1) {};
\node[vertex2] (v12b) at (0.66666*360/5:1) {};
\node[vertex] (v2) at (1*360/5:1) {};
\node[vertex2] (v23a) at (1.33333*360/5:1) {};
\node[vertex2] (v23b) at (1.66666*360/5:1) {};
\node[vertex] (v3) at (2*360/5:1) {};
\node[vertex2] (v34a) at (2.33333*360/5:1) {};
\node[vertex2] (v34b) at (2.66666*360/5:1) {};
\node[vertex] (v4) at (3*360/5:1) {};
\node[vertex2] (v45a) at (3.33333*360/5:1) {};
\node[vertex2] (v45b) at (3.66666*360/5:1) {};
\node[vertex] (v5) at (4*360/5:1) {};
\node[vertex2] (v15a) at (4.33333*360/5:1) {};
\node[vertex2] (v15b) at (4.66666*360/5:1) {};

\node[vertex2] (u1a) at (-0.3*360/5:0.5) {};
\node[vertex2] (u1b) at (0.3*360/5:0.5) {};
\node[vertex2] (u2) at (1*360/5:0.5) {};
\node[vertex2] (u3) at (2*360/5:0.5) {};
\node[vertex2] (u4a) at (2.7*360/5:0.5) {};
\node[vertex2] (u4b) at (3.3*360/5:0.5) {};
\node[vertex2] (u5) at (4*360/5:0.5) {};

\draw (v1) -- (v12a) -- (v12b) -- (v2) -- (v23a) -- (v23b) -- (v3) -- (v34a) -- (v34b) -- (v4) -- (v45a) -- (v45b) -- (v5) -- (v15a) -- (v15b) -- (v1);
\draw (v1) -- (u1a);
\draw (v1) -- (u1b);
\draw (v2) -- (u2);
\draw (v3) -- (u3);
\draw (v4) -- (u4a);
\draw (v4) -- (u4b);
\draw (v5) -- (u5);
\end{tikzpicture}
\caption{A construction of a graph $G^*$ (right) with $\rho(G^*)=\gamma(G^*)=5$ and $\gamma_t(G^*)=10$ from $G=C_5$ (left).  The graph $G^* \Box K_2$ has $\gamma_t(G^* \Box K_2)=10$ and is an instance of equality in Theorem~\ref{th:totalprodbound}.}\label{fi:HR9exam}
\end{figure}

We now derive lower bounds on $\gamma_{\times k,t}(G \Box H)$ (Theorems~\ref{totrhoprodbound} and~\ref{th:Cartrhoopensum}) in terms of the packing and open packing numbers of the graphs $G$ and $H$.

\begin{lemma}\label{rhoprodbound}
For graphs $G$ and $H$, $$\rho(G \Box H) \geq \rho(G)\rho(H).$$
\end{lemma}

\begin{proof}
Let $P_G$ and $P_H$ be maximum packings in $G$ and $H$, respectively.  It is sufficient to show that $P_G \times P_H$, which has size $\rho(G)\rho(H)$, is a packing in $G \Box H$.

If two vertices $(u,v),(x,y) \in P_G \times P_H$ are adjacent, then, by definition of a Cartesian product, either (a) $u=x$, in which case $v,y \in P_H$ are adjacent in $H$, contradicting the assumption that $P_H$ is a packing of $H$, or (b) $v=y$, in which case $u,x \in P_G$ are adjacent in $G$, contradicting the assumption that $P_G$ is a packing of $G$.

If two distinct vertices $(u,v),(x,y) \in P_G \times P_H$ have a common neighbor, $(a,b)$ say, in $G \Box H$, then by definition $$(a,b) \in \big( \overbrace{\{u\} \times N_H[v] \cup N_G[u] \times \{v\}}^{\text{closed neighborhood of $(u,v)$}} \big) \medcap \big( \overbrace{\{x\} \times N_H[y] \cup N_G[x] \times \{y\}}^{\text{closed neighborhood of $(x,y)$}} \big).$$  Four cases arise, and in each case, we contradict the assumption that $P_G$ and $P_H$ are maximum packings, tabulated below:

\begin{center}
\begin{tabular}{|c|c|c|}
\cline{2-3}
\multicolumn{1}{c|}{} & $a=u$ and $b \in N_H[v]$ & $a \in N_H[u]$ and $b=v$ \\
\hline
$a=x$ and $b \in N_H[y]$ & $\mathrm{dist}_H(v,y)=2$ & \parbox{1.8in}{\centering $\mathrm{dist}_G(u,x)=1$ \\ or \\ $u=x$ and $\mathrm{dist}_H(v,y)=1$} \\
\hline
$a \in N_G[x]$ and $b=y$ & \parbox{1.8in}{\centering $\mathrm{dist}_G(u,x)=1$ \\ or \\ $u=x$ and $\mathrm{dist}_H(v,y)=1$} & $\mathrm{dist}_G(u,x)=2$ \\
\hline
\end{tabular}.
\end{center}
\end{proof}

The following theorem follows from Lemmas~\ref{lm:tot-packing} and~\ref{rhoprodbound}.

\begin{theo}\label{totrhoprodbound}
If $G$ and $H$ are two graphs with $\delta(G)+\delta(H) \geq k$, then $$\gamma_{\times k,t}(G \Box H) \geq k \rho(G) \rho(H).$$
\end{theo}

We can also bound the open packing number of Cartesian product graphs, as in the following lemma.

\begin{lemma}\label{lm:rhoopenprodbound}
For graphs $G$ and $H$, where $G$ is not the union of disjoint $K_2$ subgraphs, $$\rho^{(\mathrm{open})}(G \Box H) \geq \rho^{(\mathrm{open})}(G)+\rho^{(\mathrm{open})}(H)-1.$$
\end{lemma}

\begin{proof}
Let $O_G$ and $O_H$ be maximal open packings in $G$ and $H$.  Choose $s \in V(G) \setminus O_G$ and $t \in O_H$.  (Note that $s$ exists because $G$ is not the union of disjoint $K_2$ subgraphs.)  An example is drawn in Figure~\ref{fi:Cartopenprodexamp}.

If $s$ is adjacent to a vertex $s' \in O_G$ in $G$, then it is adjacent to exactly one vertex in $O_G$ (since $O_G$ is an open packing).  If $s'$ exists, we define $$T=\big( \{s\} \times O_H \big) \cup \big( (O_G \setminus \{s'\}) \times \{t\} \big),$$ otherwise, we define $$T=\big( \{s\} \times O_H \big) \cup \big( O_G \times \{t\} \big).$$  Either way, $|T| \geq \rho^{(\mathrm{open})}(G)+\rho^{(\mathrm{open})}(H)-1$, so it is sufficient to show that $T$ is an open packing of $G \Box H$.  Assume, seeking a contradiction, that two distinct vertices $(u,v),(x,y) \in T$ have a common neighbor.

\begin{figure}[htp]
\centering
\begin{tikzpicture}
\matrix[nodes={draw, thick, fill=black!10, circle},row sep=0.7cm,column sep=0.4cm] {

  \node(11){}; &
  \node[star,star points=5,star point ratio=0.4](12){}; &
  \node[regular polygon,regular polygon sides=5](13){}; &
  \node(14){}; &
  \node(15){}; &
  \node[star,star points=5,star point ratio=0.4](16){}; \\
  
  \node(21){}; &
  \node(22){}; &
  \node(23){}; &
  \node(24){}; &
  \node(25){}; &
  \node(26){}; \\
  
  \node(31){}; &
  \node(32){}; &
  \node(33){}; &
  \node(34){}; &
  \node(35){}; &
  \node(36){}; \\
};

\node[draw,cross out,minimum size=11pt] at (12) {};

\draw (11) to (12) to (13) to (14) to (15) to (16); \draw (11) to[bend left] (13); \draw (14) to[bend left] (16);
\draw (21) to (22) to (23) to (24) to (25) to (26); \draw (21) to[bend left] (23); \draw (24) to[bend left] (26);
\draw (31) to (32) to (33) to (34) to (35) to (36); \draw (31) to[bend left] (33); \draw (34) to[bend left] (36);
\draw (11) to (21) to (31) to[bend left] (11);  \draw (12) to (22) to (32) to[bend left] (12);  \draw (13) to (23) to (33) to[bend left] (13);  \draw (14) to (24) to (34) to[bend left] (14);  \draw (15) to (25) to (35) to[bend left] (15);  \draw (16) to (26) to (36) to[bend left] (16);

\node at (-3.5,1.04) {$t \in V(H)  \longrightarrow$};
\node at (-0.38,1.9) {\parbox{1in}{\centering $s \in V(G) \setminus O_G$ \\ $\downarrow$}};

\end{tikzpicture}
\caption{A Cartesian product graph $G \Box H$ as in Figure~\ref{fi:packH1H2}.  The stars identify a maximum packing $O_G$ in $G$ (drawn horizontally) and the pentagon identifies a maximum packing $O_H$ in $H$ (drawn vertically).  After deleting $(s',t)$ (crossed out), we obtain an open packing of $G \Box H$.}\label{fi:Cartopenprodexamp}
\end{figure}
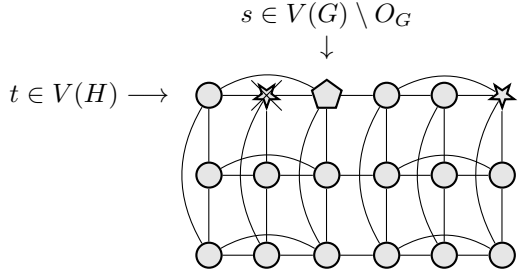

As elements of $T$, the vertices $(u,v)$ and $(x,y)$ respectively satisfy (a) either $u=s$ or $v=t$, and (b) either $x=s$ or $y=t$.  If $u=x=s$, then $v,y \in O_H$ have a common neighbor in $H$, contradicting that $O_H$ is an open packing.  A symmetric contradiction arises if $v=y=t$.  Thus, by symmetry, we can assume $u=s$ and $y=t$.

By definition of the Cartesian product, if $(s,v)$ and $(x,t)$ have a common neighbor in $G \Box H$, it is either $(s,t)$ or $(x,v)$ (or both).  Either way, we can deduce that $x$ and $s$ are adjacent in $G$.  But, since $(x,t) \in T$, we know that $x \in O_G$ or $x \in O_G \setminus \{s'\}$ (if $s'$ exists).  Either way, this is a contradiction.
\end{proof}

The following theorem follows from Lemmas~\ref{lm:tot-packing} and~\ref{lm:rhoopenprodbound}.

\begin{theo} \label{th:Cartrhoopensum}
For graphs $G$ and $H$ with $\delta(G)+\delta(H) \geq k$, where $G$ is not the union of disjoint $K_2$ subgraphs, $$\gamma_{\times k,t}(G \Box H) \geq k(\rho^{(\mathrm{open})}(G) + \rho^{(\mathrm{open})}(H) - 1).$$
\end{theo}

We also include the following, simple lower bound on $\gamma_{\times k,t}(G \Box H)$.

\begin{theo}
For graphs $G$ and $H$ with $\delta(G) \geq k$, $$\gamma_{\times k,t}(G \Box H) \leq \gamma_{\times k,t}(G)\, |V(H)|.$$
\end{theo}

\begin{proof}
A $k$TDS $D$ of $G$ gives rise to the $k$-TDS $\{(d,h):d \in D \text{ and } h \in V(H)\}$ of $G \Box H$.  An example is drawn in Figure~\ref{fi:proofexample3}.
\end{proof}

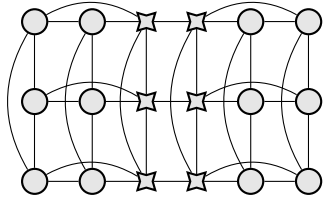
\begin{figure}[htp]
\centering
\begin{tikzpicture}
\matrix[nodes={draw, thick, fill=black!10, circle},row sep=0.7cm,column sep=0.4cm] {

  \node(11){}; &
  \node(12){}; &
  \node[star,star points=4,star point ratio=0.5](13){}; &
  \node[star,star points=4,star point ratio=0.5](14){}; &
  \node(15){}; &
  \node(16){}; \\
  
  \node(21){}; &
  \node(22){}; &
  \node[star,star points=4,star point ratio=0.5](23){}; &
  \node[star,star points=4,star point ratio=0.5](24){}; &
  \node(25){}; &
  \node(26){}; \\
  
  \node(31){}; &
  \node(32){}; &
  \node[star,star points=4,star point ratio=0.5](33){}; &
  \node[star,star points=4,star point ratio=0.5](34){}; &
  \node(35){}; &
  \node(36){}; \\
};

\draw (11) to (12) to (13) to (14) to (15) to (16); \draw (11) to[bend left] (13); \draw (14) to[bend left] (16);
\draw (21) to (22) to (23) to (24) to (25) to (26); \draw (21) to[bend left] (23); \draw (24) to[bend left] (26);
\draw (31) to (32) to (33) to (34) to (35) to (36); \draw (31) to[bend left] (33); \draw (34) to[bend left] (36);
\draw (11) to (21) to (31) to[bend left] (11);  \draw (12) to (22) to (32) to[bend left] (12);  \draw (13) to (23) to (33) to[bend left] (13);  \draw (14) to (24) to (34) to[bend left] (14);  \draw (15) to (25) to (35) to[bend left] (15);  \draw (16) to (26) to (36) to[bend left] (16);
\end{tikzpicture}
\caption{A Cartesian product graph $G \Box H$ as in Figure~\ref{fi:packH1H2}.  The stars mark a $1$TDS in each copy of $G$ (drawn horizontally), together forming a $1$TDS of $G \Box H$.}\label{fi:proofexample3}
\end{figure}

The following theorem establishes $K_n \Box K_m$, i.e., the rook's graph, as an extremal case, motivating the study of this class of graphs in the next section.

\begin{theo}\label{th:rookbound}
If $G$ and $H$ are spanning subgraphs of $G'$ and $H'$, respectively, and $\delta(G)+\delta(H) \geq k$, then $$\gamma_{\times k,t}(G \Box H) \geq \gamma_{\times k,t}(G' \Box H').$$  In particular, $$\gamma_{\times k,t}(G \Box H) \geq \gamma_{\times k,t}(K_n \Box K_m)$$ if $G$ has $n$ vertices and $H$ has $m$ vertices.
\end{theo}

\begin{proof}
Since the ``are neighbors'' vertex relationship is preserved when adding edges to a graph, any $k$TDS of $G \Box H$ remains a $k$TDS if we add edges to $G$ or $H$.
\end{proof}

The other extreme is achieved by $k$-regular graphs $G$, where $\gamma_{\times k,t}(G)=|V(G)|$, although it's not always possible to delete edges from a graph $G'$ with $\delta(G) \geq k$ to create a $k$-regular graph.

\section{The rook's graph}

In this section, we find formulas for the $k$-tuple total domination number of $K_n \Box K_m$, i.e., the $n \times m$ rook's graph.  Theorem~\ref{th:rookbound} implies that $\gamma_{\times k,t}(K_n \Box K_m)$ is an upper bound on $\gamma_{\times k,t}(G \Box H)$ when $G$ has $n$ vertices and $H$ has $m$ vertices.  Assume the vertex set of $V(K_n)$ is $\mathbb{Z}_n$.

For any $n \times m$ $(0,1)$-matrix $M=(m_{ij})$, we define $$\kappa(i,j)=\overbrace{\left( \textstyle\sum_{z \in \mathbb{Z}_m} m_{iz} \right)}^{\text{$i$-th row sum}} + \overbrace{\left( \textstyle\sum_{z \in \mathbb{Z}_n} m_{zj} \right)}^{\text{$j$-th column sum}} - 2m_{ij}.$$  A $k$TDS $D$ of $K_n \Box K_m$ corresponds to an $n \times m$ $(0,1)$-matrix $M=(m_{ij})$ with $m_{ij}=1$ if and only if $(i,j) \in D$; the matrix $M$ satisfies
\begin{equation*}
\kappa(i,j) \geq k
\end{equation*}
for all $i \in \mathbb{Z}_n$ and $j \in \mathbb{Z}_m$, which we call \emph{the $\kappa$ bound}.  We call an $n \times m$ $(0,1)$-matrix $M$ a \emph{$k$TDS matrix} if it satisfies the $\kappa$ bound for all $i \in \mathbb{Z}_n$ and $j \in \mathbb{Z}_m$.  Futher, we call $M$ a \emph{min-$k$TDS matrix} if it has exactly $\gamma_{\times k,t}(K_n \Box K_m)$ ones.  Note that a $k$TDS matrix (resp.\ min-$k$TDS matrix) remains a $k$TDS matrix (resp.\ min-$k$TDS matrix) under permutations of its rows and/or columns, and after taking its matrix transpose.

We can also interpret $(0,1)$-matrices as biadjacency matrices of bipartite graphs (since $K_n \Box K_m$ is isomorphic to the line graph of $K_{n,m}$).  Thus, a $k$TDS $D$ of $K_n \Box K_m$ also corresponds to a bipartite graph with vertex bipartition $\{R_i\}_{i \in \mathbb{Z}_n} \cup \{C_j\}_{j \in \mathbb{Z}_m}$ and an edge $R_i C_j$ whenever $(i,j) \in D$ (or equivalently whenever $m_{ij}=1$).  The bipartite graph has the property that for any pair of vertices $(R_i,C_j)$,
\begin{equation}\label{eq:bipart}
k \leq \begin{cases} \mathrm{deg}(R_i)+\mathrm{deg}(C_j)-2 & \text{if $R_i$ is adjacent to $C_j$} \\ \mathrm{deg}(R_i)+\mathrm{deg}(C_j) & \text{if $R_i$ is not adjacent to $C_j$}. \end{cases}
\end{equation}
An example of these correspondences is given in Figure~\ref{fi:rook34mat}.

\begin{figure}[htp]
\centering
\begin{tikzpicture}
\matrix[nodes={draw, thick, circle},row sep=0.7cm,column sep=0.7cm] {
  \node(11){$0$}; &
  \node(12){$1$}; &
  \node(13){$0$}; &
  \node(14){$0$}; \\
  \node(21){$1$}; &
  \node(22){$1$}; &
  \node(23){$1$}; &
  \node(24){$1$}; \\
  \node(31){$0$}; &
  \node(32){$1$}; &
  \node(33){$0$}; &
  \node(34){$0$}; \\
};
\draw (11) to (12) to (13) to (14); \draw (11) to[bend left] (13); \draw (12) to[bend left] (14); \draw (11) to[bend left] (14);
\draw (21) to (22) to (23) to (24); \draw (21) to[bend left] (23); \draw (22) to[bend left] (24); \draw (21) to[bend left] (24);
\draw (31) to (32) to (33) to (34); \draw (31) to[bend left] (33); \draw (32) to[bend left] (34); \draw (31) to[bend left] (34);
\draw (11) to (21) to (31) to[bend left] (11);  \draw (12) to (22) to (32) to[bend left] (12);  \draw (13) to (23) to (33) to[bend left] (13);  \draw (14) to (24) to (34) to[bend left] (14);
\end{tikzpicture}
\qquad\qquad\qquad
\begin{tikzpicture}
\tikzstyle{vertex}=[draw,circle,minimum size=20pt,inner sep=0pt]

\node[vertex] at (0,2) (r1) {$R_1$};
\node[vertex] at (0,1) (r2) {$R_2$};
\node[vertex] at (0,0) (r3) {$R_3$};

\node[vertex] at (2,2.5) (c1) {$C_1$};
\node[vertex] at (2,1.5) (c2) {$C_2$};
\node[vertex] at (2,0.5) (c3) {$C_3$};
\node[vertex] at (2,-0.5) (c4) {$C_4$};

\draw (r1) -- (c2);

\draw (r2) -- (c1);
\draw (r2) -- (c2);
\draw (r2) -- (c3);
\draw (r2) -- (c4);

\draw (r3) -- (c2);
\end{tikzpicture}
\caption{The $3 \times 4$ rook's graph $K_3 \Box K_4$ with the vertices in the $2$TDS in Figure~\ref{fi:rook34} labeled~$1$, illustrating the corresponding $(0,1)$-matrix and bipartite graph.}\label{fi:rook34mat}
\end{figure}
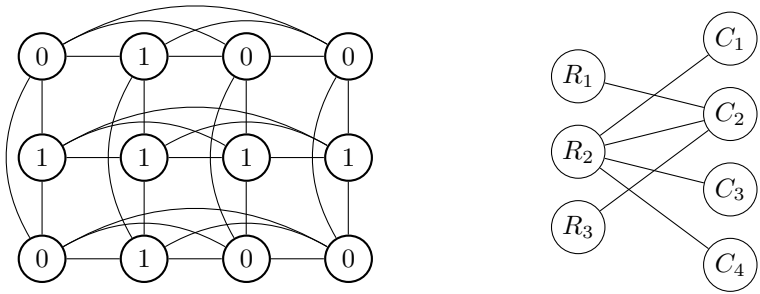

\begin{lemma}\label{lm:degsum}
When $m \geq n \geq 1$ and $n+m \neq 2$, $$\gamma_{\times k,t}(K_n \Box K_m) \geq \frac{knm}{n+m-2}.$$
\end{lemma}

\begin{proof}
For any min-$k$TDS, $\gamma_{\times k,t}(K_n \Box K_m)$ is equal to the number of edges in the corresponding bipartite graph.  We sum \eqref{eq:bipart} over all pairs of vertices $(R_i,C_j)$ to obtain $nmk \leq (n+m-2)\gamma_{\times k,t}(K_n \Box K_m)$.
\end{proof}

Lemma~\ref{lm:degsum} will be tight when the $\kappa(i,j)=k$ for all $i \in \mathbb{Z}_n$ and $j \in \mathbb{Z}_m$.  This occurs when $n=2$ and $m \geq k \geq 1$ for an $n \times m$ $(0,1)$-matrix with ones in the first $k$ columns, and zeroes elsewhere.  For example:
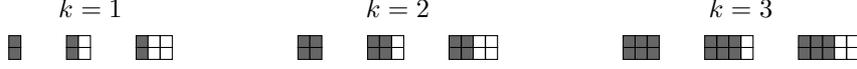
\begin{center}
$\begin{array}{ccccc}
k=1 & \qquad & k=2 & \qquad & k=3 \\
\begin{tikzpicture}
\matrix[square matrix]{
|[fill=black!60]| \\
|[fill=black!60]| \\
};
\end{tikzpicture}
\quad
\begin{tikzpicture}
\matrix[square matrix]{
|[fill=black!60]| & |[fill=white]| \\
|[fill=black!60]| & |[fill=white]| \\
};
\end{tikzpicture}
\quad
\begin{tikzpicture}
\matrix[square matrix]{
|[fill=black!60]| & |[fill=white]| & |[fill=white]| \\
|[fill=black!60]| & |[fill=white]| & |[fill=white]| \\
};
\end{tikzpicture}
& &
\begin{tikzpicture}
\matrix[square matrix]{
|[fill=black!60]| & |[fill=black!60]| \\
|[fill=black!60]| & |[fill=black!60]| \\
};
\end{tikzpicture}
\quad
\begin{tikzpicture}
\matrix[square matrix]{
|[fill=black!60]| & |[fill=black!60]| & |[fill=white]| \\
|[fill=black!60]| & |[fill=black!60]| & |[fill=white]| \\
};
\end{tikzpicture}
\quad
\begin{tikzpicture}
\matrix[square matrix]{
|[fill=black!60]| & |[fill=black!60]| & |[fill=white]| & |[fill=white]| \\
|[fill=black!60]| & |[fill=black!60]| & |[fill=white]| & |[fill=white]| \\
};
\end{tikzpicture}
& &
\begin{tikzpicture}
\matrix[square matrix]{
|[fill=black!60]| & |[fill=black!60]| & |[fill=black!60]| \\
|[fill=black!60]| & |[fill=black!60]| & |[fill=black!60]| \\
};
\end{tikzpicture}
\quad
\begin{tikzpicture}
\matrix[square matrix]{
|[fill=black!60]| & |[fill=black!60]| & |[fill=black!60]| & |[fill=white]| \\
|[fill=black!60]| & |[fill=black!60]| & |[fill=black!60]| & |[fill=white]| \\
};
\end{tikzpicture}
\quad
\begin{tikzpicture}
\matrix[square matrix]{
|[fill=black!60]| & |[fill=black!60]| & |[fill=black!60]| & |[fill=white]| & |[fill=white]| \\
|[fill=black!60]| & |[fill=black!60]| & |[fill=black!60]| & |[fill=white]| & |[fill=white]| \\
};
\end{tikzpicture}
\end{array}$
\end{center}
It is also achieved by any $n \times m$ all-$1$ matrix when $k=n+m-2$.  If equality does not hold in Lemma~\ref{lm:degsum}, then equality does not hold in the $\kappa$ bound for some cell, or equivalently, some vertex of $K_n \Box K_m$ has more than $k$ neighbors in the corresponding $k$TDS.  Of course, to have equality in Lemma~\ref{lm:degsum}, we must have $knm$ divisible by $n+m-2$.

\begin{lemma}\label{lm:sparserow}
For $n \geq 1$ and $m \geq 1$, an $n \times m$ $k$TDS matrix with an all-$0$ column has at least $kn$ ones.
\end{lemma}

\begin{proof}
If column $j^*$ is an all-$0$ column, then to achieve $\kappa(i,j^*) \geq k$ for any $i \in \mathbb{Z}_n$, we need $k$ ones in row $i$.  Since this is true for all $n$ rows, we must have $kn$ ones.
\end{proof}

There are instances when $kn$ ones is the least number of ones in any $n \times m$ $k$TDS matrix; we establish some cases in the following theorem.

\begin{theo}\label{th:manycols}
When $m \geq n \geq 2$ and $m \geq k$,
\begin{equation}\label{eq:Knmsimplebound}
\gamma_{\times k,t}(K_n \Box K_m) \leq kn
\end{equation}
with equality when $m \geq kn-1$.

When $m \geq k+1$,
\begin{equation}\label{eq:Knmsimplebound1row}
\gamma_{\times k,t}(K_1 \Box K_m)=k+1.
\end{equation}
\end{theo}

\begin{proof}
If $m \geq n \geq 2$ and $m \geq k$, the $n \times m$ $(0,1)$-matrix with ones in the first $k$ columns, and zeros elsewhere is a $k$TDS matrix, and has $kn$ ones, proving \eqref{eq:Knmsimplebound}.

Now also assume $m \geq kn-1$ and let $M$ be an $n \times m$ $k$TDS matrix.  If $M$ has a column of zeros, then $M$ has at least $kn$ ones by Lemma~\ref{lm:sparserow}.  If $M$ has no column of zeros but has at least $kn$ columns, then $M$ has at least $kn$ ones.  Thus, assume $m=kn-1$ and $M$ has a one in every column.  If $M$ has fewer than $kn$ ones, it must have exactly $1$ one in each column.  Therefore, if $m_{ij}=1$, then row $i$ must have $k+1$ ones to satisfy $\kappa(i,j) \geq k$.  If this is true for every row, then $M$ has at least $(k+1)n \geq kn$ ones.  Otherwise, there's a row of zeros, and Lemma~\ref{lm:sparserow} implies there are at least $km \geq kn$ ones.

To prove \eqref{eq:Knmsimplebound1row}, we observe that the $1 \times m$ $(0,1)$-matrix with ones in the first $k+1$ columns, and zeros elsewhere is a $k$TDS matrix, and has $k+1$ ones.  We also observe that if a $1 \times m$ $(0,1)$-matrix has fewer than $k+1$ ones, then $\kappa(i,j) \not\geq k$ for the cells $(i,j)$ containing ones, and thus is not a $k$TDS.
\end{proof}

In fact, Theorem~\ref{th:manycols} resolves the $k=1$ case since $\gamma_{\times k,t}(K_n \Box K_m)$ is undefined when $n=m=1$ (as $\delta(K_n \Box K_m)<k$).

\subsection{$2$-tuple total domination}

In this section, we derive a general formula for $\gamma_{\times 2,t}(K_n \Box K_m)$ in Proposition~\ref{pr:k2count}.  Motivated by Burchett, Lane, and Lachniet \cite{BLL}, given a $(0,1)$-matrix $M$ we construct a graph $\Gamma(M)$ with vertices corresponding to the ones in $M$, and edges between two ones belonging to the same row or column if there are no ones between them.  The following gives one such example:
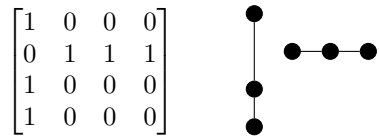
\begin{center}
\raisebox{0.3in}{
$\begin{bmatrix}
1 & 0 & 0 & 0 \\
0 & 1 & 1 & 1 \\
1 & 0 & 0 & 0 \\
1 & 0 & 0 & 0 \\
\end{bmatrix}$
}
\qquad
\begin{tikzpicture}[scale=0.5]
\tikzstyle{vertex}=[draw,circle,fill=black,minimum size=6pt,inner sep=0pt]

\node[vertex] (v1) at (0,0) {};
\node[vertex] (v2) at (0,-2) {};
\node[vertex] (v3) at (0,-3) {};

\node[vertex] (u1) at (1,-1) {};
\node[vertex] (u2) at (2,-1) {};
\node[vertex] (u3) at (3,-1) {};

\draw (v1) -- (v2) -- (v3);
\draw (u1) -- (u2) -- (u3);
\end{tikzpicture}
\end{center}
Since $2$TDS matrices $M$ correspond to graphs, we can talk about the (connected) components of $M$.  A $2$TDS matrix $M$ with a component $H$, up to permutations of the rows and colums of $M$, looks like one of the following:
$$
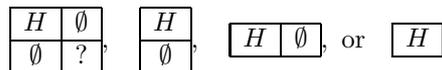
\begin{array}{|c|c|}
\hline
H & \emptyset \\
\hline
\emptyset & ? \\
\hline
\end{array},
\quad
\begin{array}{|c|}
\hline
H \\
\hline
\emptyset \\
\hline
\end{array},
\quad
\begin{array}{|c|c|}
\hline
H & \emptyset \\
\hline
\end{array},
\text{ or}
\quad
\begin{array}{|c|}
\hline
H \\
\hline
\end{array}
$$
where the question mark ($?$) denotes some $(0,1)$-submatrix, and $\emptyset$ denotes an all-$0$ submatrix.  Components of $2$TDS matrices have the following properties:
\begin{itemize}
 \item Components have no all-$0$ rows and no all-$0$ columns.
 \item Components are $2$TDS matrices in their own right.
 \item While $\Gamma(M)$ is a graph, its components arise from submatrices of $M$, so we can discuss, say, $x \times y$ components.
\end{itemize}

We will now study component switchings in $2$TDS matrices.  The following two lemmas give conditions on when some kinds of switchings are possible without increasing the number of ones (Lemma~\ref{lm:compones}) nor violating the $\kappa$ bound (Lemma~\ref{lm:switch}).

\begin{lemma}\label{lm:compones}
Let $M$ be a $2$TDS matrix and let $H$ be an $x \times y$ component of $M$.  Then the number of ones in $H$ is at least $x+y-1$.
\end{lemma}

\begin{proof}
Let $\Gamma(H)$ be the subgraph of $\Gamma(M)$ corresponding to $H$.  We choose an arbitrary vertex $v$ of $\Gamma(H)$ which has at least $2$ neighbors (since $M$ is a $2$TDS matrix) but at most $4$ neighbors (by definition of $\Gamma$).  The closed neighborhood $N_{\Gamma(H)}[v]$ has one of these properties:
\begin{itemize}
 \item It has cardinality $5$ and intersects $3$ rows and $3$ columns, and looks like the following:
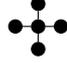
\begin{center}
\begin{tikzpicture}[scale=0.3]
\tikzstyle{vertex}=[draw,circle,fill=black,minimum size=5pt,inner sep=0pt]
\tikzstyle{vertex2}=[circle,fill=white,minimum size=5pt,inner sep=0pt]

\node[vertex] (c) at (0,0) {};

\node[vertex] (a1) at (1,0) {};
\node[vertex] (a2) at (-1,0) {};
\node[vertex] (a3) at (0,1) {};
\node[vertex] (a4) at (0,-1) {};

\draw (c) -- (a1);
\draw (c) -- (a2);
\draw (c) -- (a3);
\draw (c) -- (a4);
\end{tikzpicture}
\end{center}
 \item It has cardinality $4$ and either (a) intersects $3$ rows and $2$ columns, or (b) intersects $2$ rows and $3$ columns, and looks like one of the following:
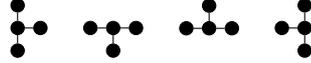
\begin{center}
\begin{tikzpicture}[scale=0.3]
\tikzstyle{vertex}=[draw,circle,fill=black,minimum size=5pt,inner sep=0pt]
\tikzstyle{vertex2}=[circle,fill=white,minimum size=5pt,inner sep=0pt]

\node[vertex] (c) at (0,0) {};

\node[vertex] (a1) at (1,0) {};
\node[vertex2] (a2) at (-1,0) {};
\node[vertex] (a3) at (0,1) {};
\node[vertex] (a4) at (0,-1) {};

\draw (c) -- (a1);
\draw (c) -- (a3);
\draw (c) -- (a4);
\end{tikzpicture}
\quad
\begin{tikzpicture}[scale=0.3]
\tikzstyle{vertex}=[draw,circle,fill=black,minimum size=5pt,inner sep=0pt]
\tikzstyle{vertex2}=[circle,fill=white,minimum size=5pt,inner sep=0pt]

\node[vertex] (c) at (0,0) {};

\node[vertex] (a1) at (1,0) {};
\node[vertex] (a2) at (-1,0) {};
\node[vertex2] (a3) at (0,1) {};
\node[vertex] (a4) at (0,-1) {};

\draw (c) -- (a1);
\draw (c) -- (a2);
\draw (c) -- (a4);
\end{tikzpicture}
\quad
\begin{tikzpicture}[scale=0.3]
\tikzstyle{vertex}=[draw,circle,fill=black,minimum size=5pt,inner sep=0pt]
\tikzstyle{vertex2}=[circle,fill=white,minimum size=5pt,inner sep=0pt]

\node[vertex] (c) at (0,0) {};

\node[vertex] (a1) at (1,0) {};
\node[vertex] (a2) at (-1,0) {};
\node[vertex] (a3) at (0,1) {};
\node[vertex2] (a4) at (0,-1) {};

\draw (c) -- (a1);
\draw (c) -- (a2);
\draw (c) -- (a3);
\end{tikzpicture}
\quad
\begin{tikzpicture}[scale=0.3]
\tikzstyle{vertex}=[draw,circle,fill=black,minimum size=5pt,inner sep=0pt]
\tikzstyle{vertex2}=[circle,fill=white,minimum size=5pt,inner sep=0pt]

\node[vertex] (c) at (0,0) {};

\node[vertex2] (a1) at (1,0) {};
\node[vertex] (a2) at (-1,0) {};
\node[vertex] (a3) at (0,1) {};
\node[vertex] (a4) at (0,-1) {};

\draw (c) -- (a2);
\draw (c) -- (a3);
\draw (c) -- (a4);
\end{tikzpicture}
\end{center}
 \item It has cardinality $3$ and either (a) intersects $1$ row and $3$ columns, (b) intersects $2$ rows and $2$ columns, or (c) intersects $3$ row and $1$ columns, and looks like one of the following:
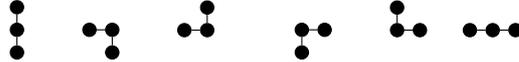
\begin{center}
\begin{tikzpicture}[scale=0.3]
\tikzstyle{vertex}=[draw,circle,fill=black,minimum size=5pt,inner sep=0pt]
\tikzstyle{vertex2}=[circle,fill=white,minimum size=5pt,inner sep=0pt]

\node[vertex] (c) at (0,0) {};

\node[vertex2] (a1) at (1,0) {};
\node[vertex2] (a2) at (-1,0) {};
\node[vertex] (a3) at (0,1) {};
\node[vertex] (a4) at (0,-1) {};

\draw (c) -- (a3);
\draw (c) -- (a4);
\end{tikzpicture}
\quad
\begin{tikzpicture}[scale=0.3]
\tikzstyle{vertex}=[draw,circle,fill=black,minimum size=5pt,inner sep=0pt]
\tikzstyle{vertex2}=[circle,fill=white,minimum size=5pt,inner sep=0pt]

\node[vertex] (c) at (0,0) {};

\node[vertex2] (a1) at (1,0) {};
\node[vertex] (a2) at (-1,0) {};
\node[vertex2] (a3) at (0,1) {};
\node[vertex] (a4) at (0,-1) {};

\draw (c) -- (a2);
\draw (c) -- (a4);
\end{tikzpicture}
\quad
\begin{tikzpicture}[scale=0.3]
\tikzstyle{vertex}=[draw,circle,fill=black,minimum size=5pt,inner sep=0pt]
\tikzstyle{vertex2}=[circle,fill=white,minimum size=5pt,inner sep=0pt]

\node[vertex] (c) at (0,0) {};

\node[vertex2] (a1) at (1,0) {};
\node[vertex] (a2) at (-1,0) {};
\node[vertex] (a3) at (0,1) {};
\node[vertex2] (a4) at (0,-1) {};

\draw (c) -- (a2);
\draw (c) -- (a3);
\end{tikzpicture}
\quad
\begin{tikzpicture}[scale=0.3]
\tikzstyle{vertex}=[draw,circle,fill=black,minimum size=5pt,inner sep=0pt]
\tikzstyle{vertex2}=[circle,fill=white,minimum size=5pt,inner sep=0pt]

\node[vertex] (c) at (0,0) {};

\node[vertex] (a1) at (1,0) {};
\node[vertex2] (a2) at (-1,0) {};
\node[vertex2] (a3) at (0,1) {};
\node[vertex] (a4) at (0,-1) {};

\draw (c) -- (a1);
\draw (c) -- (a4);
\end{tikzpicture}
\quad
\begin{tikzpicture}[scale=0.3]
\tikzstyle{vertex}=[draw,circle,fill=black,minimum size=5pt,inner sep=0pt]
\tikzstyle{vertex2}=[circle,fill=white,minimum size=5pt,inner sep=0pt]

\node[vertex] (c) at (0,0) {};

\node[vertex] (a1) at (1,0) {};
\node[vertex2] (a2) at (-1,0) {};
\node[vertex] (a3) at (0,1) {};
\node[vertex2] (a4) at (0,-1) {};

\draw (c) -- (a1);
\draw (c) -- (a3);
\end{tikzpicture}
\quad
\begin{tikzpicture}[scale=0.3]
\tikzstyle{vertex}=[draw,circle,fill=black,minimum size=5pt,inner sep=0pt]
\tikzstyle{vertex2}=[circle,fill=white,minimum size=5pt,inner sep=0pt]

\node[vertex] (c) at (0,0) {};

\node[vertex] (a1) at (1,0) {};
\node[vertex] (a2) at (-1,0) {};
\node[vertex2] (a3) at (0,1) {};
\node[vertex2] (a4) at (0,-1) {};

\draw (c) -- (a1);
\draw (c) -- (a2);
\end{tikzpicture}
\end{center}
\end{itemize}
We proceed algorithmically.  We initialize $S \gets N_{\Gamma(H)}[v]$ and iteratively add vertices to $S$ from $\Gamma(H)$ which (a) do not already belong to $S$, and (b) have a neighbor in $S$.  As a result of each iteration:
\begin{enumerate}
 \item the number of vertices in $S$ increases by exactly $1$, and
 \item one of the following:
 \begin{itemize}
  \item the number of rows of $M$ that $S$ intersects increases by exactly $1$, and the number of columns that $S$ intersects remains unchanged,
  \item the number of rows of $M$ that $S$ intersects remains unchanged, and the number of columns that $S$ intersects increases by exactly $1$, or
  \item the number of rows of $M$ that $S$ intersects remains unchanged, and the number of columns that $S$ intersects remains unchanged.
 \end{itemize}
\end{enumerate}
Since $H$ is a connected component, $\Gamma(H)$ has a spanning tree, and thus every vertex of $\Gamma(H)$ will be added to $S$ at some point.  At the end of this process $S$ intersects all $x$ rows and all $y$ columns of $H$.  The number of ones in $H$ is equal to $|S|$, which is (a) at least $5+(x-3)+(y-3)$, (b) at least $4+(x-3)+(y-2)$ or $4+(x-2)+(y-3)$, or (c) at least $3+(x-3)+(y-1)$, or $3+(x-2)+(y-2)$, or $3+(x-1)+(y-3)$.  Each of these is equal to $x+y-1$.
\end{proof}

\begin{lemma}\label{lm:switch}
Let $M$ be a $2$TDS matrix with no all-$0$ rows and no all-$0$ columns, and with an $x \times y$ union of components $K$.  Let $H$ be an $x \times y$ $2$TDS matrix with no all-$0$ rows and no all-$0$ columns.  Then replacing $K$ with $H$ in $M$ gives a $2$TDS matrix.
\end{lemma}

\begin{proof}
Call the new submatrix $\hat{M}=(\hat{m}_{ij})$.  We check the $\kappa$ bound is satisfied:
\begin{itemize}
 \item If cell $(i,j)$ is in $H$, then since $H$ is a $2$TDS matrix, the $\kappa$ bound is satisfied.
 \item If cell $(i,j)$ neither shares a row nor column with $H$, then the $\kappa$-value for $\hat{M}$ is the same as the $\kappa$-value for $M$, so the $\kappa$ bound is satisfied.
 \item If cell $(i,j)$ is in $\hat{M} \setminus H$, and shares a column (resp.\ row) with $H$, we know (i) $m_{ij}=0$, otherwise the submatrix $K$ is not the union of components, and (ii) row $i$ (resp.\ column $j$) contains a one in $\hat{M}$, (iii) column $j$ (resp.\ row $i$) contains a one in $H$.  Thus the $\kappa$ bound is satisfied. \qedhere
\end{itemize}
\end{proof}

As an example, suppose a $2$TDS matrix $M$ no all-$0$ rows nor columns contains the union of components
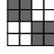
\begin{center}
\begin{tikzpicture}
\matrix[square matrix]{
|[fill=black!60]| & |[fill=black!60]| & |[fill=white]| & |[fill=white]| \\
|[fill=black!60]| & |[fill=black!60]| & |[fill=white]| & |[fill=white]| \\
|[fill=white]| & |[fill=white]| & |[fill=black!60]| & |[fill=black!60]| \\
|[fill=white]| & |[fill=white]| & |[fill=black!60]| & |[fill=black!60]| \\
};
\end{tikzpicture}
\end{center}
then we can replace it by
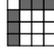
\begin{center}
\begin{tikzpicture}
\matrix[square matrix]{
|[fill=white]| & |[fill=black!60]| & |[fill=black!60]| & |[fill=black!60]| \\
|[fill=black!60]| & |[fill=white]| & |[fill=white]| & |[fill=white]| \\
|[fill=black!60]| & |[fill=white]| & |[fill=white]| & |[fill=white]| \\
|[fill=black!60]| & |[fill=white]| & |[fill=white]| & |[fill=white]| \\
};
\end{tikzpicture}
\end{center}
and Lemma~\ref{lm:switch} implies that the matrix obtained after performing this switch is also a $2$TDS matrix.  Furthermore, since this switch decreases the number of ones, we deduce that the original matrix $M$ is not a min-$2$TDS matrix.

As another example, if a $2$TDS matrix $M$ no all-$0$ rows nor columns contains the component
\begin{center}
\begin{tikzpicture}
\matrix[square matrix]{
|[fill=black!60]| & |[fill=black!60]| & |[fill=white]| & |[fill=white]| \\
|[fill=black!60]| & |[fill=white]| & |[fill=black!60]| & |[fill=white]| \\
|[fill=white]| & |[fill=white]| & |[fill=black!60]| & |[fill=black!60]| \\
};
\end{tikzpicture}
\end{center}
we can replace it by 
\begin{center}
\begin{tikzpicture}
\matrix[square matrix]{
|[fill=black!60]| & |[fill=black!60]| & |[fill=black!60]| & |[fill=black!60]| \\
|[fill=black!60]| & |[fill=white]| & |[fill=white]| & |[fill=white]| \\
|[fill=black!60]| & |[fill=white]| & |[fill=white]| & |[fill=white]| \\
};
\end{tikzpicture}
\end{center}
and Lemma~\ref{lm:switch} implies that we obtain a $2$TDS matrix.  Moreover, since the number of ones is unchanged after performing this switch, if $M$ is a min-$2$TDS matrix, then we obtain another min-$2$TDS matrix after switching.  Since Lemma~\ref{lm:compones} implies that any $3 \times 4$ component of $M$ has at least $6$ ones, we can replace every $3 \times 4$ component in this way while still preserving the min-$2$TDS property, thereby reducing the possibilities we need to consider.

In the subsequent material, switchings as per Lemma~\ref{lm:switch} will arise repeatedly, and we will not indicate its use each time.

Lemmas~\ref{lm:compones} and~\ref{lm:switch} are the primary motivation for the next theorem (Theorem~\ref{th:compclass}).  We will repeatedly use the following $(0,1)$-matrices, which we give notation to:  For $x \geq 1$ and $y \geq 1$, we define $J(x,y)$ as the $x \times y$ all-$1$ matrix.  For $y \geq 6$, we define the $2 \times y$ matrix $A(y)$ to have the first row $(1,1,1,0,0,\ldots,0)$ and second row $(0,0,0,1,1,\ldots,1)$, depicted below for $y \in \{6,7,8,9,10\}$:
\begin{center}
\begin{tikzpicture}
\matrix[square matrix]{
|[fill=black!60]| & |[fill=black!60]| & |[fill=black!60]| & |[fill=white]| & |[fill=white]| & |[fill=white]| \\
|[fill=white]| & |[fill=white]| & |[fill=white]| & |[fill=black!60]| & |[fill=black!60]| & |[fill=black!60]| \\
};
\end{tikzpicture}
\quad
\begin{tikzpicture}
\matrix[square matrix]{
|[fill=black!60]| & |[fill=black!60]| & |[fill=black!60]| & |[fill=white]| & |[fill=white]| & |[fill=white]| & |[fill=white]| \\
|[fill=white]| & |[fill=white]| & |[fill=white]| & |[fill=black!60]| & |[fill=black!60]| & |[fill=black!60]| & |[fill=black!60]| \\
};
\end{tikzpicture}
\quad
\begin{tikzpicture}
\matrix[square matrix]{
|[fill=black!60]| & |[fill=black!60]| & |[fill=black!60]| & |[fill=white]| & |[fill=white]| & |[fill=white]| & |[fill=white]| & |[fill=white]| \\
|[fill=white]| & |[fill=white]| & |[fill=white]| & |[fill=black!60]| & |[fill=black!60]| & |[fill=black!60]| & |[fill=black!60]| & |[fill=black!60]| \\
};
\end{tikzpicture}
\quad
\begin{tikzpicture}
\matrix[square matrix]{
|[fill=black!60]| & |[fill=black!60]| & |[fill=black!60]| & |[fill=white]| & |[fill=white]| & |[fill=white]| & |[fill=white]| & |[fill=white]| & |[fill=white]| \\
|[fill=white]| & |[fill=white]| & |[fill=white]| & |[fill=black!60]| & |[fill=black!60]| & |[fill=black!60]| & |[fill=black!60]| & |[fill=black!60]| & |[fill=black!60]| \\
};
\end{tikzpicture}
\quad
\begin{tikzpicture}
\matrix[square matrix]{
|[fill=black!60]| & |[fill=black!60]| & |[fill=black!60]| & |[fill=white]| & |[fill=white]| & |[fill=white]| & |[fill=white]| & |[fill=white]| & |[fill=white]| & |[fill=white]| \\
|[fill=white]| & |[fill=white]| & |[fill=white]| & |[fill=black!60]| & |[fill=black!60]| & |[fill=black!60]| & |[fill=black!60]| & |[fill=black!60]| & |[fill=black!60]| & |[fill=black!60]| \\
};
\end{tikzpicture}
\end{center}
For $x \geq 3$ and $y \geq 3$, let $B(x,y)$ be the $x \times y$ $(0,1)$ matrix with an all-$1$ first row, an all-$1$ first column, and zeroes elsewhere, depicted below for $x \in \{3,4,5\}$ and $y \in \{3,4,5,6\}$:
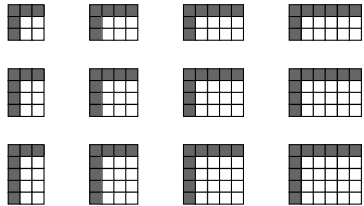
\begin{center}
$\begin{array}{cccc}
\begin{tikzpicture}
\matrix[square matrix]{
|[fill=black!60]| & |[fill=black!60]| & |[fill=black!60]| \\
|[fill=black!60]| & |[fill=white]| & |[fill=white]| \\
|[fill=black!60]| & |[fill=white]| & |[fill=white]| \\
};
\end{tikzpicture}
&
\begin{tikzpicture}
\matrix[square matrix]{
|[fill=black!60]| & |[fill=black!60]| & |[fill=black!60]| & |[fill=black!60]| \\
|[fill=black!60]| & |[fill=white]| & |[fill=white]| & |[fill=white]| \\
|[fill=black!60]| & |[fill=white]| & |[fill=white]| & |[fill=white]| \\
};
\end{tikzpicture}
&
\begin{tikzpicture}
\matrix[square matrix]{
|[fill=black!60]| & |[fill=black!60]| & |[fill=black!60]| & |[fill=black!60]| & |[fill=black!60]| \\
|[fill=black!60]| & |[fill=white]| & |[fill=white]| & |[fill=white]| & |[fill=white]| \\
|[fill=black!60]| & |[fill=white]| & |[fill=white]| & |[fill=white]| & |[fill=white]| \\
};
\end{tikzpicture}
&
\begin{tikzpicture}
\matrix[square matrix]{
|[fill=black!60]| & |[fill=black!60]| & |[fill=black!60]| & |[fill=black!60]| & |[fill=black!60]| & |[fill=black!60]| \\
|[fill=black!60]| & |[fill=white]| & |[fill=white]| & |[fill=white]| & |[fill=white]| & |[fill=white]| \\
|[fill=black!60]| & |[fill=white]| & |[fill=white]| & |[fill=white]| & |[fill=white]| & |[fill=white]| \\
};
\end{tikzpicture}
\\
\begin{tikzpicture}
\matrix[square matrix]{
|[fill=black!60]| & |[fill=black!60]| & |[fill=black!60]| \\
|[fill=black!60]| & |[fill=white]| & |[fill=white]| \\
|[fill=black!60]| & |[fill=white]| & |[fill=white]| \\
|[fill=black!60]| & |[fill=white]| & |[fill=white]| \\
};
\end{tikzpicture}
&
\begin{tikzpicture}
\matrix[square matrix]{
|[fill=black!60]| & |[fill=black!60]| & |[fill=black!60]| & |[fill=black!60]| \\
|[fill=black!60]| & |[fill=white]| & |[fill=white]| & |[fill=white]| \\
|[fill=black!60]| & |[fill=white]| & |[fill=white]| & |[fill=white]| \\
|[fill=black!60]| & |[fill=white]| & |[fill=white]| & |[fill=white]| \\
};
\end{tikzpicture}
&
\begin{tikzpicture}
\matrix[square matrix]{
|[fill=black!60]| & |[fill=black!60]| & |[fill=black!60]| & |[fill=black!60]| & |[fill=black!60]| \\
|[fill=black!60]| & |[fill=white]| & |[fill=white]| & |[fill=white]| & |[fill=white]| \\
|[fill=black!60]| & |[fill=white]| & |[fill=white]| & |[fill=white]| & |[fill=white]| \\
|[fill=black!60]| & |[fill=white]| & |[fill=white]| & |[fill=white]| & |[fill=white]| \\
};
\end{tikzpicture}
&
\begin{tikzpicture}
\matrix[square matrix]{
|[fill=black!60]| & |[fill=black!60]| & |[fill=black!60]| & |[fill=black!60]| & |[fill=black!60]| & |[fill=black!60]| \\
|[fill=black!60]| & |[fill=white]| & |[fill=white]| & |[fill=white]| & |[fill=white]| & |[fill=white]| \\
|[fill=black!60]| & |[fill=white]| & |[fill=white]| & |[fill=white]| & |[fill=white]| & |[fill=white]| \\
|[fill=black!60]| & |[fill=white]| & |[fill=white]| & |[fill=white]| & |[fill=white]| & |[fill=white]| \\
};
\end{tikzpicture}
\\
\begin{tikzpicture}
\matrix[square matrix]{
|[fill=black!60]| & |[fill=black!60]| & |[fill=black!60]| \\
|[fill=black!60]| & |[fill=white]| & |[fill=white]| \\
|[fill=black!60]| & |[fill=white]| & |[fill=white]| \\
|[fill=black!60]| & |[fill=white]| & |[fill=white]| \\
|[fill=black!60]| & |[fill=white]| & |[fill=white]| \\
};
\end{tikzpicture}
&
\begin{tikzpicture}
\matrix[square matrix]{
|[fill=black!60]| & |[fill=black!60]| & |[fill=black!60]| & |[fill=black!60]| \\
|[fill=black!60]| & |[fill=white]| & |[fill=white]| & |[fill=white]| \\
|[fill=black!60]| & |[fill=white]| & |[fill=white]| & |[fill=white]| \\
|[fill=black!60]| & |[fill=white]| & |[fill=white]| & |[fill=white]| \\
|[fill=black!60]| & |[fill=white]| & |[fill=white]| & |[fill=white]| \\
};
\end{tikzpicture}
&
\begin{tikzpicture}
\matrix[square matrix]{
|[fill=black!60]| & |[fill=black!60]| & |[fill=black!60]| & |[fill=black!60]| & |[fill=black!60]| \\
|[fill=black!60]| & |[fill=white]| & |[fill=white]| & |[fill=white]| & |[fill=white]| \\
|[fill=black!60]| & |[fill=white]| & |[fill=white]| & |[fill=white]| & |[fill=white]| \\
|[fill=black!60]| & |[fill=white]| & |[fill=white]| & |[fill=white]| & |[fill=white]| \\
|[fill=black!60]| & |[fill=white]| & |[fill=white]| & |[fill=white]| & |[fill=white]| \\
};
\end{tikzpicture}
&
\begin{tikzpicture}
\matrix[square matrix]{
|[fill=black!60]| & |[fill=black!60]| & |[fill=black!60]| & |[fill=black!60]| & |[fill=black!60]| & |[fill=black!60]| \\
|[fill=black!60]| & |[fill=white]| & |[fill=white]| & |[fill=white]| & |[fill=white]| & |[fill=white]| \\
|[fill=black!60]| & |[fill=white]| & |[fill=white]| & |[fill=white]| & |[fill=white]| & |[fill=white]| \\
|[fill=black!60]| & |[fill=white]| & |[fill=white]| & |[fill=white]| & |[fill=white]| & |[fill=white]| \\
|[fill=black!60]| & |[fill=white]| & |[fill=white]| & |[fill=white]| & |[fill=white]| & |[fill=white]| \\
};
\end{tikzpicture}
\end{array}$
\end{center}
For $x \geq 4$ and $y \geq 4$, let $C(x,y)$ be the $x \times y$ $(0,1)$ matrix first row $(0,1,1,\ldots,1)$, first column $(0,1,1,\ldots,1)^T$, and zeroes elsewhere, depicted below for $x \in \{4,5,6\}$ and $y \in \{4,5,6,7\}$:
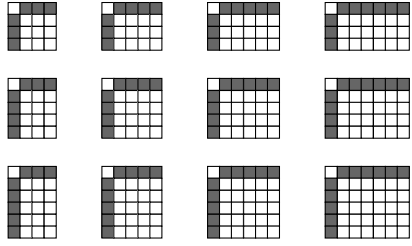
\begin{center}
$\begin{array}{cccc}
\begin{tikzpicture}
\matrix[square matrix]{
|[fill=white]| & |[fill=black!60]| & |[fill=black!60]| & |[fill=black!60]| \\
|[fill=black!60]| & |[fill=white]| & |[fill=white]| & |[fill=white]| \\
|[fill=black!60]| & |[fill=white]| & |[fill=white]| & |[fill=white]| \\
|[fill=black!60]| & |[fill=white]| & |[fill=white]| & |[fill=white]| \\
};
\end{tikzpicture}
&
\begin{tikzpicture}
\matrix[square matrix]{
|[fill=white]| & |[fill=black!60]| & |[fill=black!60]| & |[fill=black!60]| & |[fill=black!60]| \\
|[fill=black!60]| & |[fill=white]| & |[fill=white]| & |[fill=white]| & |[fill=white]| \\
|[fill=black!60]| & |[fill=white]| & |[fill=white]| & |[fill=white]| & |[fill=white]| \\
|[fill=black!60]| & |[fill=white]| & |[fill=white]| & |[fill=white]| & |[fill=white]| \\
};
\end{tikzpicture}
&
\begin{tikzpicture}
\matrix[square matrix]{
|[fill=white]| & |[fill=black!60]| & |[fill=black!60]| & |[fill=black!60]| & |[fill=black!60]| & |[fill=black!60]| \\
|[fill=black!60]| & |[fill=white]| & |[fill=white]| & |[fill=white]| & |[fill=white]| & |[fill=white]| \\
|[fill=black!60]| & |[fill=white]| & |[fill=white]| & |[fill=white]| & |[fill=white]| & |[fill=white]| \\
|[fill=black!60]| & |[fill=white]| & |[fill=white]| & |[fill=white]| & |[fill=white]| & |[fill=white]| \\
};
\end{tikzpicture}
&
\begin{tikzpicture}
\matrix[square matrix]{
|[fill=white]| & |[fill=black!60]| & |[fill=black!60]| & |[fill=black!60]| & |[fill=black!60]| & |[fill=black!60]| & |[fill=black!60]| \\
|[fill=black!60]| & |[fill=white]| & |[fill=white]| & |[fill=white]| & |[fill=white]| & |[fill=white]| & |[fill=white]| \\
|[fill=black!60]| & |[fill=white]| & |[fill=white]| & |[fill=white]| & |[fill=white]| & |[fill=white]| & |[fill=white]| \\
|[fill=black!60]| & |[fill=white]| & |[fill=white]| & |[fill=white]| & |[fill=white]| & |[fill=white]| & |[fill=white]| \\
};
\end{tikzpicture}
\\
\begin{tikzpicture}
\matrix[square matrix]{
|[fill=white]| & |[fill=black!60]| & |[fill=black!60]| & |[fill=black!60]| \\
|[fill=black!60]| & |[fill=white]| & |[fill=white]| & |[fill=white]| \\
|[fill=black!60]| & |[fill=white]| & |[fill=white]| & |[fill=white]| \\
|[fill=black!60]| & |[fill=white]| & |[fill=white]| & |[fill=white]| \\
|[fill=black!60]| & |[fill=white]| & |[fill=white]| & |[fill=white]| \\
};
\end{tikzpicture}
&
\begin{tikzpicture}
\matrix[square matrix]{
|[fill=white]| & |[fill=black!60]| & |[fill=black!60]| & |[fill=black!60]| & |[fill=black!60]| \\
|[fill=black!60]| & |[fill=white]| & |[fill=white]| & |[fill=white]| & |[fill=white]| \\
|[fill=black!60]| & |[fill=white]| & |[fill=white]| & |[fill=white]| & |[fill=white]| \\
|[fill=black!60]| & |[fill=white]| & |[fill=white]| & |[fill=white]| & |[fill=white]| \\
|[fill=black!60]| & |[fill=white]| & |[fill=white]| & |[fill=white]| & |[fill=white]| \\
};
\end{tikzpicture}
&
\begin{tikzpicture}
\matrix[square matrix]{
|[fill=white]| & |[fill=black!60]| & |[fill=black!60]| & |[fill=black!60]| & |[fill=black!60]| & |[fill=black!60]| \\
|[fill=black!60]| & |[fill=white]| & |[fill=white]| & |[fill=white]| & |[fill=white]| & |[fill=white]| \\
|[fill=black!60]| & |[fill=white]| & |[fill=white]| & |[fill=white]| & |[fill=white]| & |[fill=white]| \\
|[fill=black!60]| & |[fill=white]| & |[fill=white]| & |[fill=white]| & |[fill=white]| & |[fill=white]| \\
|[fill=black!60]| & |[fill=white]| & |[fill=white]| & |[fill=white]| & |[fill=white]| & |[fill=white]| \\
};
\end{tikzpicture}
&
\begin{tikzpicture}
\matrix[square matrix]{
|[fill=white]| & |[fill=black!60]| & |[fill=black!60]| & |[fill=black!60]| & |[fill=black!60]| & |[fill=black!60]| & |[fill=black!60]| \\
|[fill=black!60]| & |[fill=white]| & |[fill=white]| & |[fill=white]| & |[fill=white]| & |[fill=white]| & |[fill=white]| \\
|[fill=black!60]| & |[fill=white]| & |[fill=white]| & |[fill=white]| & |[fill=white]| & |[fill=white]| & |[fill=white]| \\
|[fill=black!60]| & |[fill=white]| & |[fill=white]| & |[fill=white]| & |[fill=white]| & |[fill=white]| & |[fill=white]| \\
|[fill=black!60]| & |[fill=white]| & |[fill=white]| & |[fill=white]| & |[fill=white]| & |[fill=white]| & |[fill=white]| \\
};
\end{tikzpicture}
\\
\begin{tikzpicture}
\matrix[square matrix]{
|[fill=white]| & |[fill=black!60]| & |[fill=black!60]| & |[fill=black!60]| \\
|[fill=black!60]| & |[fill=white]| & |[fill=white]| & |[fill=white]| \\
|[fill=black!60]| & |[fill=white]| & |[fill=white]| & |[fill=white]| \\
|[fill=black!60]| & |[fill=white]| & |[fill=white]| & |[fill=white]| \\
|[fill=black!60]| & |[fill=white]| & |[fill=white]| & |[fill=white]| \\
|[fill=black!60]| & |[fill=white]| & |[fill=white]| & |[fill=white]| \\
};
\end{tikzpicture}
&
\begin{tikzpicture}
\matrix[square matrix]{
|[fill=white]| & |[fill=black!60]| & |[fill=black!60]| & |[fill=black!60]| & |[fill=black!60]| \\
|[fill=black!60]| & |[fill=white]| & |[fill=white]| & |[fill=white]| & |[fill=white]| \\
|[fill=black!60]| & |[fill=white]| & |[fill=white]| & |[fill=white]| & |[fill=white]| \\
|[fill=black!60]| & |[fill=white]| & |[fill=white]| & |[fill=white]| & |[fill=white]| \\
|[fill=black!60]| & |[fill=white]| & |[fill=white]| & |[fill=white]| & |[fill=white]| \\
|[fill=black!60]| & |[fill=white]| & |[fill=white]| & |[fill=white]| & |[fill=white]| \\
};
\end{tikzpicture}
&
\begin{tikzpicture}
\matrix[square matrix]{
|[fill=white]| & |[fill=black!60]| & |[fill=black!60]| & |[fill=black!60]| & |[fill=black!60]| & |[fill=black!60]| \\
|[fill=black!60]| & |[fill=white]| & |[fill=white]| & |[fill=white]| & |[fill=white]| & |[fill=white]| \\
|[fill=black!60]| & |[fill=white]| & |[fill=white]| & |[fill=white]| & |[fill=white]| & |[fill=white]| \\
|[fill=black!60]| & |[fill=white]| & |[fill=white]| & |[fill=white]| & |[fill=white]| & |[fill=white]| \\
|[fill=black!60]| & |[fill=white]| & |[fill=white]| & |[fill=white]| & |[fill=white]| & |[fill=white]| \\
|[fill=black!60]| & |[fill=white]| & |[fill=white]| & |[fill=white]| & |[fill=white]| & |[fill=white]| \\
};
\end{tikzpicture}
&
\begin{tikzpicture}
\matrix[square matrix]{
|[fill=white]| & |[fill=black!60]| & |[fill=black!60]| & |[fill=black!60]| & |[fill=black!60]| & |[fill=black!60]| & |[fill=black!60]| \\
|[fill=black!60]| & |[fill=white]| & |[fill=white]| & |[fill=white]| & |[fill=white]| & |[fill=white]| & |[fill=white]| \\
|[fill=black!60]| & |[fill=white]| & |[fill=white]| & |[fill=white]| & |[fill=white]| & |[fill=white]| & |[fill=white]| \\
|[fill=black!60]| & |[fill=white]| & |[fill=white]| & |[fill=white]| & |[fill=white]| & |[fill=white]| & |[fill=white]| \\
|[fill=black!60]| & |[fill=white]| & |[fill=white]| & |[fill=white]| & |[fill=white]| & |[fill=white]| & |[fill=white]| \\
|[fill=black!60]| & |[fill=white]| & |[fill=white]| & |[fill=white]| & |[fill=white]| & |[fill=white]| & |[fill=white]| \\
};
\end{tikzpicture}
\\
\end{array}$
\end{center}

\begin{theo}\label{th:compclass}
For $n \geq 1$ and $m \geq 1$, excluding $(n,m) \in \{(1,1),(1,2),(2,1)\}$, there exists an $n \times m$ min-$2$TDS matrix $M$ whose components, up to permutations of the rows and columns, are all either $J(x,1)$ for $x \geq 3$ (or $J(1,y)$ for $y \geq 3$), or $J(x,2)$ for $x \geq 2$ (or $J(2,y)$ for $y \geq 2$), or $M=B(3,3)$.
\end{theo}

\begin{proof}
We start with a min-$2$TDS matrix $M$.  (Such a matrix does not exist when $(n,m) \in \{(1,1),(1,2),(2,1)\}$, since $\delta(K_n \Box K_m) \leq 1 < 2$).

\textit{Case I}:  $M$ has an all-$0$ column (or, by symmetry, an all-$0$ row).  Then $M$ has at least $2n$ ones by Lemma~\ref{th:manycols}, which is the same number of ones as $J(n,2)$ with $m-2$ appended columns of zeros, in which case the theorem is true.  We henceforth assume $M$ has no all-$0$ rows nor all-$0$ columns.

\textit{Case II}: $M$ has a $2 \times y$ component with $y \geq 6$ (or its transpose).  We replace it by $A(y)$, but since $A(y)$ has $y$ ones, whereas any $2 \times y$ component has at least $y+1$ ones by Lemma~\ref{lm:compones}, we contradict the assumption that $M$ is a min-$2$TDS matrix.

\textit{Case III}: $M$ has a $x \times y$ component with $x \geq 3$ and $y \geq 3$.  We replace it by $B(x,y)$.  Lemma~\ref{lm:compones} implies that the number of ones has not increased, so we still have a min-$2$TDS matrix.

\textit{Case IV}: $M$ has a $2 \times y$ component $H$ with $2 \leq y \leq 5$ (or its transpose).  There must be an all-$1$ column for the component to be connected.  Hence $H$ has at least $y+1$ ones.
\begin{itemize}
 \item If there are no other components, then $M=H=J(2,2)$, as the other possible $2 \times y$ components are not min-$2$TDSs.
 \item If there is at least one other component $K$, then:
 \begin{itemize}
  \item If $H$ has at least two all-$1$ columns, then, since $2 \leq y \leq 5$, it is equivalent to one of the following:
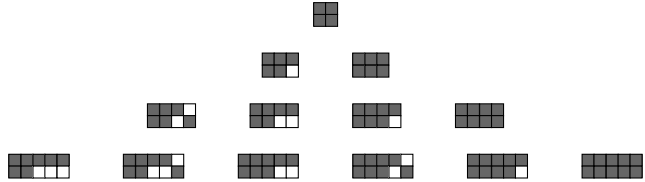
\begin{center}
\begin{tikzpicture}
\matrix[square matrix]{
|[fill=black!60]| & |[fill=black!60]| \\
|[fill=black!60]| & |[fill=black!60]| \\
};
\end{tikzpicture}

\begin{tikzpicture}
\matrix[square matrix]{
|[fill=black!60]| & |[fill=black!60]| & |[fill=black!60]| \\
|[fill=black!60]| & |[fill=black!60]| & |[fill=white]| \\
};
\end{tikzpicture}
\quad
\begin{tikzpicture}
\matrix[square matrix]{
|[fill=black!60]| & |[fill=black!60]| & |[fill=black!60]| \\
|[fill=black!60]| & |[fill=black!60]| & |[fill=black!60]| \\
};
\end{tikzpicture}

\begin{tikzpicture}
\matrix[square matrix]{
|[fill=black!60]| & |[fill=black!60]| & |[fill=black!60]| & |[fill=white]| \\
|[fill=black!60]| & |[fill=black!60]| & |[fill=white]| & |[fill=black!60]| \\
};
\end{tikzpicture}
\quad
\begin{tikzpicture}
\matrix[square matrix]{
|[fill=black!60]| & |[fill=black!60]| & |[fill=black!60]| & |[fill=black!60]| \\
|[fill=black!60]| & |[fill=black!60]| & |[fill=white]| & |[fill=white]| \\
};
\end{tikzpicture}
\quad
\begin{tikzpicture}
\matrix[square matrix]{
|[fill=black!60]| & |[fill=black!60]| & |[fill=black!60]| & |[fill=black!60]| \\
|[fill=black!60]| & |[fill=black!60]| & |[fill=black!60]| & |[fill=white]| \\
};
\end{tikzpicture}
\quad
\begin{tikzpicture}
\matrix[square matrix]{
|[fill=black!60]| & |[fill=black!60]| & |[fill=black!60]| & |[fill=black!60]| \\
|[fill=black!60]| & |[fill=black!60]| & |[fill=black!60]| & |[fill=black!60]| \\
};
\end{tikzpicture}

\begin{tikzpicture}
\matrix[square matrix]{
|[fill=black!60]| & |[fill=black!60]| & |[fill=black!60]| & |[fill=black!60]| & |[fill=black!60]| \\
|[fill=black!60]| & |[fill=black!60]| & |[fill=white]| & |[fill=white]| & |[fill=white]| \\
};
\end{tikzpicture}
\quad
\begin{tikzpicture}
\matrix[square matrix]{
|[fill=black!60]| & |[fill=black!60]| & |[fill=black!60]| & |[fill=black!60]| & |[fill=white]| \\
|[fill=black!60]| & |[fill=black!60]| & |[fill=white]| & |[fill=white]| & |[fill=black!60]| \\
};
\end{tikzpicture}
\quad
\begin{tikzpicture}
\matrix[square matrix]{
|[fill=black!60]| & |[fill=black!60]| & |[fill=black!60]| & |[fill=black!60]| & |[fill=black!60]| \\
|[fill=black!60]| & |[fill=black!60]| & |[fill=black!60]| & |[fill=white]| & |[fill=white]| \\
};
\end{tikzpicture}
\quad
\begin{tikzpicture}
\matrix[square matrix]{
|[fill=black!60]| & |[fill=black!60]| & |[fill=black!60]| & |[fill=black!60]| & |[fill=white]| \\
|[fill=black!60]| & |[fill=black!60]| & |[fill=black!60]| & |[fill=white]| & |[fill=black!60]| \\
};
\end{tikzpicture}
\quad
\begin{tikzpicture}
\matrix[square matrix]{
|[fill=black!60]| & |[fill=black!60]| & |[fill=black!60]| & |[fill=black!60]| & |[fill=black!60]| \\
|[fill=black!60]| & |[fill=black!60]| & |[fill=black!60]| & |[fill=black!60]| & |[fill=white]| \\
};
\end{tikzpicture}
\quad
\begin{tikzpicture}
\matrix[square matrix]{
|[fill=black!60]| & |[fill=black!60]| & |[fill=black!60]| & |[fill=black!60]| & |[fill=black!60]| \\
|[fill=black!60]| & |[fill=black!60]| & |[fill=black!60]| & |[fill=black!60]| & |[fill=black!60]| \\
};
\end{tikzpicture}
\end{center}
In this case, the union $H \cup K$ thus has dimensions $x' \times y'$ where $x' \geq 3$ and $y' \geq 3$ and at least $x'+y'-1$ ones.  We replace this union of components with $B(x',y')$, which has $x'+y'-1$ ones.
  \item If $H$ has exactly one all-$1$ column, then, since $y \leq 5$ and $H$ is a $2$TDS, it is equivalent to:
\begin{center}
\begin{tikzpicture}
\matrix[square matrix]{
|[fill=black!60]| & |[fill=black!60]| & |[fill=black!60]| & |[fill=white]| & |[fill=white]| \\
|[fill=black!60]| & |[fill=white]| & |[fill=white]| & |[fill=black!60]| & |[fill=black!60]| \\
};
\end{tikzpicture}
\end{center}
 \begin{itemize}
  \item If the components of $M$ except for $H$ are all $J(1,3)$ matrices, then $M$ has more than $2n$ ones, whence Theorem~\ref{th:manycols} contradicts the assumption that $M$ is a min-$2$TDS matrix.
  \item Otherwise, we can choose $K$ to have at least $2$ rows.  The union $H \cup K$ has dimensions $x' \times y'$ where $x' \geq 4$ and $y' \geq 4$ and at least $x'+y'-2$ ones by Lemma~\ref{lm:compones}.  We replace $H \cup K$ by $C(x',y')$.
 \end{itemize}
\end{itemize}
\end{itemize}

\textit{Case V}: $M$ has a $3 \times y$ component with $y \geq 9$ (or its transpose).  As a result of Case~III, components have the form $B(3,y)$, and we perform the switches indicated below:
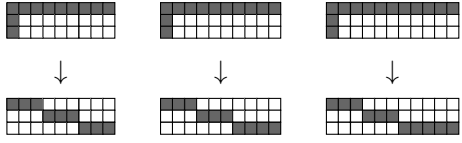
\begin{center}
$
\begin{array}{ccc}
\begin{tikzpicture}
\matrix[square matrix]{
|[fill=black!60]| & |[fill=black!60]| & |[fill=black!60]| & |[fill=black!60]| & |[fill=black!60]| & |[fill=black!60]| & |[fill=black!60]| & |[fill=black!60]| & |[fill=black!60]| \\
|[fill=black!60]| & |[fill=white]| & |[fill=white]| & |[fill=white]| & |[fill=white]| & |[fill=white]| & |[fill=white]| & |[fill=white]| & |[fill=white]| \\
|[fill=black!60]| & |[fill=white]| & |[fill=white]| & |[fill=white]| & |[fill=white]| & |[fill=white]| & |[fill=white]| & |[fill=white]| & |[fill=white]| \\
};
\end{tikzpicture}
&
\begin{tikzpicture}
\matrix[square matrix]{
|[fill=black!60]| & |[fill=black!60]| & |[fill=black!60]| & |[fill=black!60]| & |[fill=black!60]| & |[fill=black!60]| & |[fill=black!60]| & |[fill=black!60]| & |[fill=black!60]| & |[fill=black!60]| \\
|[fill=black!60]| & |[fill=white]| & |[fill=white]| & |[fill=white]| & |[fill=white]| & |[fill=white]| & |[fill=white]| & |[fill=white]| & |[fill=white]| & |[fill=white]| \\
|[fill=black!60]| & |[fill=white]| & |[fill=white]| & |[fill=white]| & |[fill=white]| & |[fill=white]| & |[fill=white]| & |[fill=white]| & |[fill=white]| & |[fill=white]| \\
};
\end{tikzpicture}
&
\begin{tikzpicture}
\matrix[square matrix]{
|[fill=black!60]| & |[fill=black!60]| & |[fill=black!60]| & |[fill=black!60]| & |[fill=black!60]| & |[fill=black!60]| & |[fill=black!60]| & |[fill=black!60]| & |[fill=black!60]| & |[fill=black!60]| & |[fill=black!60]| \\
|[fill=black!60]| & |[fill=white]| & |[fill=white]| & |[fill=white]| & |[fill=white]| & |[fill=white]| & |[fill=white]| & |[fill=white]| & |[fill=white]| & |[fill=white]| & |[fill=white]| \\
|[fill=black!60]| & |[fill=white]| & |[fill=white]| & |[fill=white]| & |[fill=white]| & |[fill=white]| & |[fill=white]| & |[fill=white]| & |[fill=white]| & |[fill=white]| & |[fill=white]| \\
};
\end{tikzpicture}
\\
\downarrow
&
\downarrow
&
\downarrow
\\
\begin{tikzpicture}
\matrix[square matrix]{
|[fill=black!60]| & |[fill=black!60]| & |[fill=black!60]| & |[fill=white]| & |[fill=white]| & |[fill=white]| & |[fill=white]| & |[fill=white]| & |[fill=white]| \\
|[fill=white]| & |[fill=white]| & |[fill=white]| & |[fill=black!60]| & |[fill=black!60]| & |[fill=black!60]| & |[fill=white]| & |[fill=white]| & |[fill=white]| \\
|[fill=white]| & |[fill=white]| & |[fill=white]| & |[fill=white]| & |[fill=white]| & |[fill=white]| & |[fill=black!60]| & |[fill=black!60]| & |[fill=black!60]| \\
};
\end{tikzpicture}
&
\begin{tikzpicture}
\matrix[square matrix]{
|[fill=black!60]| & |[fill=black!60]| & |[fill=black!60]| & |[fill=white]| & |[fill=white]| & |[fill=white]| & |[fill=white]| & |[fill=white]| & |[fill=white]| & |[fill=white]| \\
|[fill=white]| & |[fill=white]| & |[fill=white]| & |[fill=black!60]| & |[fill=black!60]| & |[fill=black!60]| & |[fill=white]| & |[fill=white]| & |[fill=white]| & |[fill=white]| \\
|[fill=white]| & |[fill=white]| & |[fill=white]| & |[fill=white]| & |[fill=white]| & |[fill=white]| & |[fill=black!60]| & |[fill=black!60]| & |[fill=black!60]| & |[fill=black!60]| \\
};
\end{tikzpicture}
&
\begin{tikzpicture}
\matrix[square matrix]{
|[fill=black!60]| & |[fill=black!60]| & |[fill=black!60]| & |[fill=white]| & |[fill=white]| & |[fill=white]| & |[fill=white]| & |[fill=white]| & |[fill=white]| & |[fill=white]| & |[fill=white]| \\
|[fill=white]| & |[fill=white]| & |[fill=white]| & |[fill=black!60]| & |[fill=black!60]| & |[fill=black!60]| & |[fill=white]| & |[fill=white]| & |[fill=white]| & |[fill=white]| & |[fill=white]| \\
|[fill=white]| & |[fill=white]| & |[fill=white]| & |[fill=white]| & |[fill=white]| & |[fill=white]| & |[fill=black!60]| & |[fill=black!60]| & |[fill=black!60]| & |[fill=black!60]| & |[fill=black!60]| \\
};
\end{tikzpicture}
\end{array}
$
\end{center}
and so on.  By Lemma~\ref{lm:switch} we obtain a $2$TDS matrix with fewer ones than $M$, giving a contradiction.

\textit{Case VI}: $M$ has a $3 \times y$ component with $3 \leq y \leq 8$ (or its transpose).  As a result of Case~III, components have the form $B(3,y)$:
\begin{center}
\begin{tikzpicture}
\matrix[square matrix]{
|[fill=black!60]| & |[fill=black!60]| & |[fill=black!60]| \\
|[fill=black!60]| & |[fill=white]| & |[fill=white]| \\
|[fill=black!60]| & |[fill=white]| & |[fill=white]| \\
};
\end{tikzpicture}
\quad
\begin{tikzpicture}
\matrix[square matrix]{
|[fill=black!60]| & |[fill=black!60]| & |[fill=black!60]| & |[fill=black!60]| \\
|[fill=black!60]| & |[fill=white]| & |[fill=white]| & |[fill=white]| \\
|[fill=black!60]| & |[fill=white]| & |[fill=white]| & |[fill=white]| \\
};
\end{tikzpicture}
\quad
\begin{tikzpicture}
\matrix[square matrix]{
|[fill=black!60]| & |[fill=black!60]| & |[fill=black!60]| & |[fill=black!60]| & |[fill=black!60]| \\
|[fill=black!60]| & |[fill=white]| & |[fill=white]| & |[fill=white]| & |[fill=white]| \\
|[fill=black!60]| & |[fill=white]| & |[fill=white]| & |[fill=white]| & |[fill=white]| \\
};
\end{tikzpicture}
\quad
\begin{tikzpicture}
\matrix[square matrix]{
|[fill=black!60]| & |[fill=black!60]| & |[fill=black!60]| & |[fill=black!60]| & |[fill=black!60]| & |[fill=black!60]| \\
|[fill=black!60]| & |[fill=white]| & |[fill=white]| & |[fill=white]| & |[fill=white]| & |[fill=white]| \\
|[fill=black!60]| & |[fill=white]| & |[fill=white]| & |[fill=white]| & |[fill=white]| & |[fill=white]| \\
};
\end{tikzpicture}
\quad
\begin{tikzpicture}
\matrix[square matrix]{
|[fill=black!60]| & |[fill=black!60]| & |[fill=black!60]| & |[fill=black!60]| & |[fill=black!60]| & |[fill=black!60]| & |[fill=black!60]| \\
|[fill=black!60]| & |[fill=white]| & |[fill=white]| & |[fill=white]| & |[fill=white]| & |[fill=white]| & |[fill=white]| \\
|[fill=black!60]| & |[fill=white]| & |[fill=white]| & |[fill=white]| & |[fill=white]| & |[fill=white]| & |[fill=white]| \\
};
\end{tikzpicture}
\quad
\begin{tikzpicture}
\matrix[square matrix]{
|[fill=black!60]| & |[fill=black!60]| & |[fill=black!60]| & |[fill=black!60]| & |[fill=black!60]| & |[fill=black!60]| & |[fill=black!60]| & |[fill=black!60]| \\
|[fill=black!60]| & |[fill=white]| & |[fill=white]| & |[fill=white]| & |[fill=white]| & |[fill=white]| & |[fill=white]| & |[fill=white]| \\
|[fill=black!60]| & |[fill=white]| & |[fill=white]| & |[fill=white]| & |[fill=white]| & |[fill=white]| & |[fill=white]| & |[fill=white]| \\
};
\end{tikzpicture}
\end{center}
If there are no other components, then $M=B(3,y)$.  However, $B(3,y)$ is not a min-$2$TDS for $5 \leq y \leq 8$,  so $M \in \{B(3,3), B(3,4)\}$ (or $M=B(3,4)^T=B(4,3)$).  However, If $M=B(3,4)$, we can replace it with
\begin{center}
\begin{tikzpicture}
\matrix[square matrix]{
|[fill=black!60]| & |[fill=black!60]| & |[fill=white]| & |[fill=white]| \\
|[fill=black!60]| & |[fill=black!60]| & |[fill=white]| & |[fill=white]| \\
|[fill=black!60]| & |[fill=black!60]| & |[fill=white]| & |[fill=white]| \\
};
\end{tikzpicture}
\end{center}
which has the same number of ones as $B(3,4)$ (and likewise if $M=B(4,3)$).  If there is another component $K$, then $B(3,y) \cup K$ has dimensions $x' \times y'$ where $x' \geq 4$ and $y' \geq 4$ and at least $x'+y'-2$ ones.  We replace $B(3,y) \cup K$ with $C(x',y')$, which has exactly $x'+y'-2$ ones.

\textit{Case VII}: $M$ has a $x \times y$ component with $x \geq 4$ and $y \geq 4$.  As a result of Case~III, components have the form $B(x,y)$, which we can replace by $C(x,y)$ to obtain a $2$TDS matrix with fewer ones than $M$, giving a contradiction.
\end{proof}

Theorem~\ref{th:compclass} implies that, for all $n \geq 1$ and $m \geq 1$ except when $(n,m) \in \{(1,1),(1,2),(2,1)\}$, there is some min-$2$TDS matrix whose components belong only to a strongly restricted family of components.  In the next theorem, we restrict this family of component further when considering matrices with no all-$0$ rows and no all-$0$ columns.

\begin{theo}\label{th:compclass2}
For $n \geq 1$ and $m \geq 1$ except $(n,m) \in \{(1,1),(1,2),(2,1),(2,2),(3,3),(3,4),(4,3)\}$, if there exists an $n \times m$ min-$2$TDS matrix with no all-$0$ rows and no all-$0$ columns, then there exists an $n \times m$ min-$2$TDS matrix $M$ whose components, up to permutations of the rows and columns, are all $J(1,3)$ or $J(3,1)$, except possibly for
\begin{itemize}
 \item exactly one $J(x,1)$ component with $4 \leq x \leq 7$;
 \item exactly one $J(1,y)$ component with $4 \leq y \leq 7$; or
 \item exactly one $J(x,1)$ component with $4 \leq x \leq 5$ and exactly one $J(1,y)$ component with $4 \leq y \leq 5$.
\end{itemize}
Further, the number of ones in $M$ is
$$\begin{cases}
  \lceil 3(n+m)/4 \rceil +1 & \text{ if $m \equiv 3n+4 \pmod 8$} \\
  \lceil 3(n+m)/4 \rceil & \text{ otherwise.}
  \end{cases}$$
\end{theo}

\begin{proof}
The proof of Theorem~\ref{th:compclass} implies we can assume that each component of $M$ has one of the forms: $J(x,1)$ for $x \geq 3$, or $J(1,y)$ for $y \geq 3$.

\textit{Case I}: $M$ has two components $J(1,y)$ and $J(1,y')$ with $y \geq 4$ and $y' \geq 4$.  (Or, by symmetry, $M$ has two components $J(x,1)$ and $J(x',1)$ with $x \geq 4$ and $x' \geq 4$.)  We replace them by the two components $J(1,3)$ and $J(1,y+y'-3)$.  An example is drawn below when $y=4$ and $y'=5$:
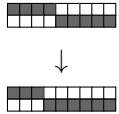
\begin{center}
$
\begin{array}{c}
\begin{tikzpicture}
\matrix[square matrix]{
|[fill=black!60]| & |[fill=black!60]| & |[fill=black!60]| & |[fill=black!60]| & |[fill=white]| & |[fill=white]| & |[fill=white]| & |[fill=white]| & |[fill=white]| \\
|[fill=white]| & |[fill=white]| & |[fill=white]| & |[fill=white]| & |[fill=black!60]| & |[fill=black!60]| & |[fill=black!60]| & |[fill=black!60]| & |[fill=black!60]| \\
};
\end{tikzpicture}
\\
\downarrow
\\
\begin{tikzpicture}
\matrix[square matrix]{
|[fill=black!60]| & |[fill=black!60]| & |[fill=black!60]| & |[fill=white]| & |[fill=white]| & |[fill=white]| & |[fill=white]| & |[fill=white]| & |[fill=white]| \\
|[fill=white]| & |[fill=white]| & |[fill=white]| & |[fill=black!60]| & |[fill=black!60]| & |[fill=black!60]| & |[fill=black!60]| & |[fill=black!60]| & |[fill=black!60]| \\
};
\end{tikzpicture}
\end{array}
$
\end{center}
We repeat this process until there is at most one component of the form $J(1,y)$ with $y \geq 4$, and likewise at most one component of the form $J(x,1)$ with $x \geq 4$.

\textit{Case II}: $M$ has a $J(1,y)$ component with $y \geq 6$ and a $J(x,1)$ component with $x \geq 4$.  (Or, by symmetry, $M$ has a $J(x,1)$ component with $x \geq 6$ and a $J(1,y)$ component with $y \geq 4$.)  We apply the switch indicated below:
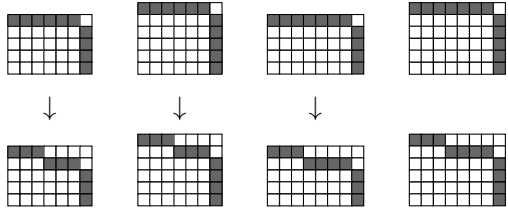
\begin{center}
$
\begin{array}{cccc}
\begin{tikzpicture}
\matrix[square matrix]{
|[fill=black!60]| & |[fill=black!60]| & |[fill=black!60]| & |[fill=black!60]| & |[fill=black!60]| & |[fill=black!60]| & |[fill=white]| \\
|[fill=white]| & |[fill=white]| & |[fill=white]| & |[fill=white]| & |[fill=white]| & |[fill=white]| & |[fill=black!60]| \\
|[fill=white]| & |[fill=white]| & |[fill=white]| & |[fill=white]| & |[fill=white]| & |[fill=white]| & |[fill=black!60]| \\
|[fill=white]| & |[fill=white]| & |[fill=white]| & |[fill=white]| & |[fill=white]| & |[fill=white]| & |[fill=black!60]| \\
|[fill=white]| & |[fill=white]| & |[fill=white]| & |[fill=white]| & |[fill=white]| & |[fill=white]| & |[fill=black!60]| \\
};
\end{tikzpicture}
&
\begin{tikzpicture}
\matrix[square matrix]{
|[fill=black!60]| & |[fill=black!60]| & |[fill=black!60]| & |[fill=black!60]| & |[fill=black!60]| & |[fill=black!60]| & |[fill=white]| \\
|[fill=white]| & |[fill=white]| & |[fill=white]| & |[fill=white]| & |[fill=white]| & |[fill=white]| & |[fill=black!60]| \\
|[fill=white]| & |[fill=white]| & |[fill=white]| & |[fill=white]| & |[fill=white]| & |[fill=white]| & |[fill=black!60]| \\
|[fill=white]| & |[fill=white]| & |[fill=white]| & |[fill=white]| & |[fill=white]| & |[fill=white]| & |[fill=black!60]| \\
|[fill=white]| & |[fill=white]| & |[fill=white]| & |[fill=white]| & |[fill=white]| & |[fill=white]| & |[fill=black!60]| \\
|[fill=white]| & |[fill=white]| & |[fill=white]| & |[fill=white]| & |[fill=white]| & |[fill=white]| & |[fill=black!60]| \\
};
\end{tikzpicture}
&
\begin{tikzpicture}
\matrix[square matrix]{
|[fill=black!60]| & |[fill=black!60]| & |[fill=black!60]| & |[fill=black!60]| & |[fill=black!60]| & |[fill=black!60]| & |[fill=black!60]| & |[fill=white]| \\
|[fill=white]| & |[fill=white]| & |[fill=white]| & |[fill=white]| & |[fill=white]| & |[fill=white]| & |[fill=white]| & |[fill=black!60]| \\
|[fill=white]| & |[fill=white]| & |[fill=white]| & |[fill=white]| & |[fill=white]| & |[fill=white]| & |[fill=white]| & |[fill=black!60]| \\
|[fill=white]| & |[fill=white]| & |[fill=white]| & |[fill=white]| & |[fill=white]| & |[fill=white]| & |[fill=white]| & |[fill=black!60]| \\
|[fill=white]| & |[fill=white]| & |[fill=white]| & |[fill=white]| & |[fill=white]| & |[fill=white]| & |[fill=white]| & |[fill=black!60]| \\
};
\end{tikzpicture}
&
\begin{tikzpicture}
\matrix[square matrix]{
|[fill=black!60]| & |[fill=black!60]| & |[fill=black!60]| & |[fill=black!60]| & |[fill=black!60]| & |[fill=black!60]| & |[fill=black!60]| & |[fill=white]| \\
|[fill=white]| & |[fill=white]| & |[fill=white]| & |[fill=white]| & |[fill=white]| & |[fill=white]| & |[fill=white]| & |[fill=black!60]| \\
|[fill=white]| & |[fill=white]| & |[fill=white]| & |[fill=white]| & |[fill=white]| & |[fill=white]| & |[fill=white]| & |[fill=black!60]| \\
|[fill=white]| & |[fill=white]| & |[fill=white]| & |[fill=white]| & |[fill=white]| & |[fill=white]| & |[fill=white]| & |[fill=black!60]| \\
|[fill=white]| & |[fill=white]| & |[fill=white]| & |[fill=white]| & |[fill=white]| & |[fill=white]| & |[fill=white]| & |[fill=black!60]| \\
|[fill=white]| & |[fill=white]| & |[fill=white]| & |[fill=white]| & |[fill=white]| & |[fill=white]| & |[fill=white]| & |[fill=black!60]| \\
};
\end{tikzpicture}
\\
\downarrow
&
\downarrow
&
\downarrow
\\
\begin{tikzpicture}
\matrix[square matrix]{
|[fill=black!60]| & |[fill=black!60]| & |[fill=black!60]| & |[fill=white]| & |[fill=white]| & |[fill=white]| & |[fill=white]| \\
|[fill=white]| & |[fill=white]| & |[fill=white]| & |[fill=black!60]| & |[fill=black!60]| & |[fill=black!60]| & |[fill=white]| \\
|[fill=white]| & |[fill=white]| & |[fill=white]| & |[fill=white]| & |[fill=white]| & |[fill=white]| & |[fill=black!60]| \\
|[fill=white]| & |[fill=white]| & |[fill=white]| & |[fill=white]| & |[fill=white]| & |[fill=white]| & |[fill=black!60]| \\
|[fill=white]| & |[fill=white]| & |[fill=white]| & |[fill=white]| & |[fill=white]| & |[fill=white]| & |[fill=black!60]| \\
};
\end{tikzpicture}
&
\begin{tikzpicture}
\matrix[square matrix]{
|[fill=black!60]| & |[fill=black!60]| & |[fill=black!60]| & |[fill=white]| & |[fill=white]| & |[fill=white]| & |[fill=white]| \\
|[fill=white]| & |[fill=white]| & |[fill=white]| & |[fill=black!60]| & |[fill=black!60]| & |[fill=black!60]| & |[fill=white]| \\
|[fill=white]| & |[fill=white]| & |[fill=white]| & |[fill=white]| & |[fill=white]| & |[fill=white]| & |[fill=black!60]| \\
|[fill=white]| & |[fill=white]| & |[fill=white]| & |[fill=white]| & |[fill=white]| & |[fill=white]| & |[fill=black!60]| \\
|[fill=white]| & |[fill=white]| & |[fill=white]| & |[fill=white]| & |[fill=white]| & |[fill=white]| & |[fill=black!60]| \\
|[fill=white]| & |[fill=white]| & |[fill=white]| & |[fill=white]| & |[fill=white]| & |[fill=white]| & |[fill=black!60]| \\
};
\end{tikzpicture}
&
\begin{tikzpicture}
\matrix[square matrix]{
|[fill=black!60]| & |[fill=black!60]| & |[fill=black!60]| & |[fill=white]| & |[fill=white]| & |[fill=white]| & |[fill=white]| & |[fill=white]| \\
|[fill=white]| & |[fill=white]| & |[fill=white]| & |[fill=black!60]| & |[fill=black!60]| & |[fill=black!60]| & |[fill=black!60]| & |[fill=white]| \\
|[fill=white]| & |[fill=white]| & |[fill=white]| & |[fill=white]| & |[fill=white]| & |[fill=white]| & |[fill=white]| & |[fill=black!60]| \\
|[fill=white]| & |[fill=white]| & |[fill=white]| & |[fill=white]| & |[fill=white]| & |[fill=white]| & |[fill=white]| & |[fill=black!60]| \\
|[fill=white]| & |[fill=white]| & |[fill=white]| & |[fill=white]| & |[fill=white]| & |[fill=white]| & |[fill=white]| & |[fill=black!60]| \\
};
\end{tikzpicture}
&
\begin{tikzpicture}
\matrix[square matrix]{
|[fill=black!60]| & |[fill=black!60]| & |[fill=black!60]| & |[fill=white]| & |[fill=white]| & |[fill=white]| & |[fill=white]| & |[fill=white]| \\
|[fill=white]| & |[fill=white]| & |[fill=white]| & |[fill=black!60]| & |[fill=black!60]| & |[fill=black!60]| & |[fill=black!60]| & |[fill=white]| \\
|[fill=white]| & |[fill=white]| & |[fill=white]| & |[fill=white]| & |[fill=white]| & |[fill=white]| & |[fill=white]| & |[fill=black!60]| \\
|[fill=white]| & |[fill=white]| & |[fill=white]| & |[fill=white]| & |[fill=white]| & |[fill=white]| & |[fill=white]| & |[fill=black!60]| \\
|[fill=white]| & |[fill=white]| & |[fill=white]| & |[fill=white]| & |[fill=white]| & |[fill=white]| & |[fill=white]| & |[fill=black!60]| \\
|[fill=white]| & |[fill=white]| & |[fill=white]| & |[fill=white]| & |[fill=white]| & |[fill=white]| & |[fill=white]| & |[fill=black!60]| \\
};
\end{tikzpicture}
\end{array}
$
\end{center}
and so on in other cases.  This reduces the number of ones, contradicting that $M$ is a min-$2$TDS matrix.

\textit{Case III}: $M$ has a $J(1,y)$ component with $y \geq 10$.  (Or, by symmetry, $M$ has a $J(x,1)$ component with $x \geq 10$.)  There must also be a $J(3,1)$ component in $M$, otherwise every row contains at least $3$ ones, contradicting that $M$ is a min-$2$TDS matrix.  We apply the switch indicated below:
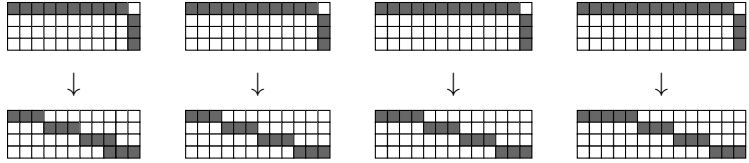
\begin{center}
$
\begin{array}{cccc}
\begin{tikzpicture}
\matrix[square matrix]{
|[fill=black!60]| & |[fill=black!60]| & |[fill=black!60]| & |[fill=black!60]| & |[fill=black!60]| & |[fill=black!60]| & |[fill=black!60]| & |[fill=black!60]| & |[fill=black!60]| & |[fill=black!60]| & |[fill=white]| \\
|[fill=white]| & |[fill=white]| & |[fill=white]| & |[fill=white]| & |[fill=white]| & |[fill=white]| & |[fill=white]| & |[fill=white]| & |[fill=white]| & |[fill=white]| & |[fill=black!60]| \\
|[fill=white]| & |[fill=white]| & |[fill=white]| & |[fill=white]| & |[fill=white]| & |[fill=white]| & |[fill=white]| & |[fill=white]| & |[fill=white]| & |[fill=white]| & |[fill=black!60]| \\
|[fill=white]| & |[fill=white]| & |[fill=white]| & |[fill=white]| & |[fill=white]| & |[fill=white]| & |[fill=white]| & |[fill=white]| & |[fill=white]| & |[fill=white]| & |[fill=black!60]| \\
};
\end{tikzpicture}
&
\begin{tikzpicture}
\matrix[square matrix]{
|[fill=black!60]| & |[fill=black!60]| & |[fill=black!60]| & |[fill=black!60]| & |[fill=black!60]| & |[fill=black!60]| & |[fill=black!60]| & |[fill=black!60]| & |[fill=black!60]| & |[fill=black!60]| & |[fill=black!60]| & |[fill=white]| \\
|[fill=white]| & |[fill=white]| & |[fill=white]| & |[fill=white]| & |[fill=white]| & |[fill=white]| & |[fill=white]| & |[fill=white]| & |[fill=white]| & |[fill=white]| & |[fill=white]| & |[fill=black!60]| \\
|[fill=white]| & |[fill=white]| & |[fill=white]| & |[fill=white]| & |[fill=white]| & |[fill=white]| & |[fill=white]| & |[fill=white]| & |[fill=white]| & |[fill=white]| & |[fill=white]| & |[fill=black!60]| \\
|[fill=white]| & |[fill=white]| & |[fill=white]| & |[fill=white]| & |[fill=white]| & |[fill=white]| & |[fill=white]| & |[fill=white]| & |[fill=white]| & |[fill=white]| & |[fill=white]| & |[fill=black!60]| \\
};
\end{tikzpicture}
&
\begin{tikzpicture}
\matrix[square matrix]{
|[fill=black!60]| & |[fill=black!60]| & |[fill=black!60]| & |[fill=black!60]| & |[fill=black!60]| & |[fill=black!60]| & |[fill=black!60]| & |[fill=black!60]| & |[fill=black!60]| & |[fill=black!60]| & |[fill=black!60]| & |[fill=black!60]| & |[fill=white]| \\
|[fill=white]| & |[fill=white]| & |[fill=white]| & |[fill=white]| & |[fill=white]| & |[fill=white]| & |[fill=white]| & |[fill=white]| & |[fill=white]| & |[fill=white]| & |[fill=white]| & |[fill=white]| & |[fill=black!60]| \\
|[fill=white]| & |[fill=white]| & |[fill=white]| & |[fill=white]| & |[fill=white]| & |[fill=white]| & |[fill=white]| & |[fill=white]| & |[fill=white]| & |[fill=white]| & |[fill=white]| & |[fill=white]| & |[fill=black!60]| \\
|[fill=white]| & |[fill=white]| & |[fill=white]| & |[fill=white]| & |[fill=white]| & |[fill=white]| & |[fill=white]| & |[fill=white]| & |[fill=white]| & |[fill=white]| & |[fill=white]| & |[fill=white]| & |[fill=black!60]| \\
};
\end{tikzpicture}
&
\begin{tikzpicture}
\matrix[square matrix]{
|[fill=black!60]| & |[fill=black!60]| & |[fill=black!60]| & |[fill=black!60]| & |[fill=black!60]| & |[fill=black!60]| & |[fill=black!60]| & |[fill=black!60]| & |[fill=black!60]| & |[fill=black!60]| & |[fill=black!60]| & |[fill=black!60]| & |[fill=black!60]| & |[fill=white]| \\
|[fill=white]| & |[fill=white]| & |[fill=white]| & |[fill=white]| & |[fill=white]| & |[fill=white]| & |[fill=white]| & |[fill=white]| & |[fill=white]| & |[fill=white]| & |[fill=white]| & |[fill=white]| & |[fill=white]| & |[fill=black!60]| \\
|[fill=white]| & |[fill=white]| & |[fill=white]| & |[fill=white]| & |[fill=white]| & |[fill=white]| & |[fill=white]| & |[fill=white]| & |[fill=white]| & |[fill=white]| & |[fill=white]| & |[fill=white]| & |[fill=white]| & |[fill=black!60]| \\
|[fill=white]| & |[fill=white]| & |[fill=white]| & |[fill=white]| & |[fill=white]| & |[fill=white]| & |[fill=white]| & |[fill=white]| & |[fill=white]| & |[fill=white]| & |[fill=white]| & |[fill=white]| & |[fill=white]| & |[fill=black!60]| \\
};
\end{tikzpicture}
\\
\downarrow
&
\downarrow
&
\downarrow
&
\downarrow
\\
\begin{tikzpicture}
\matrix[square matrix]{
|[fill=black!60]| & |[fill=black!60]| & |[fill=black!60]| & |[fill=white]| & |[fill=white]| & |[fill=white]| & |[fill=white]| & |[fill=white]| & |[fill=white]| & |[fill=white]| & |[fill=white]| \\
|[fill=white]| & |[fill=white]| & |[fill=white]| & |[fill=black!60]| & |[fill=black!60]| & |[fill=black!60]| & |[fill=white]| & |[fill=white]| & |[fill=white]| & |[fill=white]| & |[fill=white]| \\
|[fill=white]| & |[fill=white]| & |[fill=white]| & |[fill=white]| & |[fill=white]| & |[fill=white]| & |[fill=black!60]| & |[fill=black!60]| & |[fill=black!60]| & |[fill=white]| & |[fill=white]| \\
|[fill=white]| & |[fill=white]| & |[fill=white]| & |[fill=white]| & |[fill=white]| & |[fill=white]| & |[fill=white]| & |[fill=white]| & |[fill=black!60]| & |[fill=black!60]| & |[fill=black!60]| \\
};
\end{tikzpicture}
&
\begin{tikzpicture}
\matrix[square matrix]{
|[fill=black!60]| & |[fill=black!60]| & |[fill=black!60]| & |[fill=white]| & |[fill=white]| & |[fill=white]| & |[fill=white]| & |[fill=white]| & |[fill=white]| & |[fill=white]| & |[fill=white]| & |[fill=white]| \\
|[fill=white]| & |[fill=white]| & |[fill=white]| & |[fill=black!60]| & |[fill=black!60]| & |[fill=black!60]| & |[fill=white]| & |[fill=white]| & |[fill=white]| & |[fill=white]| & |[fill=white]| & |[fill=white]| \\
|[fill=white]| & |[fill=white]| & |[fill=white]| & |[fill=white]| & |[fill=white]| & |[fill=white]| & |[fill=black!60]| & |[fill=black!60]| & |[fill=black!60]| & |[fill=white]| & |[fill=white]| & |[fill=white]| \\
|[fill=white]| & |[fill=white]| & |[fill=white]| & |[fill=white]| & |[fill=white]| & |[fill=white]| & |[fill=white]| & |[fill=white]| & |[fill=white]| & |[fill=black!60]| & |[fill=black!60]| & |[fill=black!60]| \\
};
\end{tikzpicture}
&
\begin{tikzpicture}
\matrix[square matrix]{
|[fill=black!60]| & |[fill=black!60]| & |[fill=black!60]| & |[fill=black!60]| & |[fill=white]| & |[fill=white]| & |[fill=white]| & |[fill=white]| & |[fill=white]| & |[fill=white]| & |[fill=white]| & |[fill=white]| & |[fill=white]| \\
|[fill=white]| & |[fill=white]| & |[fill=white]| & |[fill=white]| & |[fill=black!60]| & |[fill=black!60]| & |[fill=black!60]| & |[fill=white]| & |[fill=white]| & |[fill=white]| & |[fill=white]| & |[fill=white]| & |[fill=white]| \\
|[fill=white]| & |[fill=white]| & |[fill=white]| & |[fill=white]| & |[fill=white]| & |[fill=white]| & |[fill=white]| & |[fill=black!60]| & |[fill=black!60]| & |[fill=black!60]| & |[fill=white]| & |[fill=white]| & |[fill=white]| \\
|[fill=white]| & |[fill=white]| & |[fill=white]| & |[fill=white]| & |[fill=white]| & |[fill=white]| & |[fill=white]| & |[fill=white]| & |[fill=white]| & |[fill=white]| & |[fill=black!60]| & |[fill=black!60]| & |[fill=black!60]| \\
};
\end{tikzpicture}
&
\begin{tikzpicture}
\matrix[square matrix]{
|[fill=black!60]| & |[fill=black!60]| & |[fill=black!60]| & |[fill=black!60]| & |[fill=black!60]| & |[fill=white]| & |[fill=white]| & |[fill=white]| & |[fill=white]| & |[fill=white]| & |[fill=white]| & |[fill=white]| & |[fill=white]| & |[fill=white]| \\
|[fill=white]| & |[fill=white]| & |[fill=white]| & |[fill=white]| & |[fill=white]| & |[fill=black!60]| & |[fill=black!60]| & |[fill=black!60]| & |[fill=white]| & |[fill=white]| & |[fill=white]| & |[fill=white]| & |[fill=white]| & |[fill=white]| \\
|[fill=white]| & |[fill=white]| & |[fill=white]| & |[fill=white]| & |[fill=white]| & |[fill=white]| & |[fill=white]| & |[fill=white]| & |[fill=black!60]| & |[fill=black!60]| & |[fill=black!60]| & |[fill=white]| & |[fill=white]| & |[fill=white]| \\
|[fill=white]| & |[fill=white]| & |[fill=white]| & |[fill=white]| & |[fill=white]| & |[fill=white]| & |[fill=white]| & |[fill=white]| & |[fill=white]| & |[fill=white]| & |[fill=white]| & |[fill=black!60]| & |[fill=black!60]| & |[fill=black!60]| \\
};
\end{tikzpicture}
\end{array}
$
\end{center}
and so on in other cases.  These switches all reduce the number of ones, contradicting that $M$ is a min-$2$TDS matrix.

\textit{Case III}: $M$ has a $J(1,y)$ component with $8 \leq y \leq 9$.  (Or, by symmetry, $M$ has a $J(x,1)$ component with $8 \leq x \leq 9$.)  There must also be at least two $J(3,1)$ components in $M$, otherwise the average number of ones per row is at least $(8+3)/4>2$, contradicting that $M$ is a min-$2$TDS matrix.  We apply the switch indicated below:
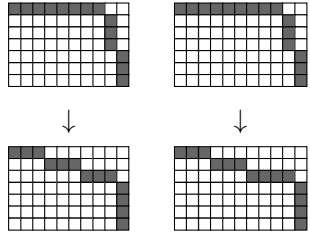
\begin{center}
$\begin{array}{cc}
\begin{tikzpicture}
\matrix[square matrix]{
|[fill=black!60]| & |[fill=black!60]| & |[fill=black!60]| & |[fill=black!60]| & |[fill=black!60]| & |[fill=black!60]| & |[fill=black!60]| & |[fill=black!60]| & |[fill=white]| & |[fill=white]| \\
|[fill=white]| & |[fill=white]| & |[fill=white]| & |[fill=white]| & |[fill=white]| & |[fill=white]| & |[fill=white]| & |[fill=white]| & |[fill=black!60]| & |[fill=white]| \\
|[fill=white]| & |[fill=white]| & |[fill=white]| & |[fill=white]| & |[fill=white]| & |[fill=white]| & |[fill=white]| & |[fill=white]| & |[fill=black!60]| & |[fill=white]| \\
|[fill=white]| & |[fill=white]| & |[fill=white]| & |[fill=white]| & |[fill=white]| & |[fill=white]| & |[fill=white]| & |[fill=white]| & |[fill=black!60]| & |[fill=white]| \\
|[fill=white]| & |[fill=white]| & |[fill=white]| & |[fill=white]| & |[fill=white]| & |[fill=white]| & |[fill=white]| & |[fill=white]| & |[fill=white]| & |[fill=black!60]| \\
|[fill=white]| & |[fill=white]| & |[fill=white]| & |[fill=white]| & |[fill=white]| & |[fill=white]| & |[fill=white]| & |[fill=white]| & |[fill=white]| & |[fill=black!60]| \\
|[fill=white]| & |[fill=white]| & |[fill=white]| & |[fill=white]| & |[fill=white]| & |[fill=white]| & |[fill=white]| & |[fill=white]| & |[fill=white]| & |[fill=black!60]| \\
};
\end{tikzpicture}
&
\begin{tikzpicture}
\matrix[square matrix]{
|[fill=black!60]| & |[fill=black!60]| & |[fill=black!60]| & |[fill=black!60]| & |[fill=black!60]| & |[fill=black!60]| & |[fill=black!60]| & |[fill=black!60]| & |[fill=black!60]| & |[fill=white]| & |[fill=white]| \\
|[fill=white]| & |[fill=white]| & |[fill=white]| & |[fill=white]| & |[fill=white]| & |[fill=white]| & |[fill=white]| & |[fill=white]| & |[fill=white]| & |[fill=black!60]| & |[fill=white]| \\
|[fill=white]| & |[fill=white]| & |[fill=white]| & |[fill=white]| & |[fill=white]| & |[fill=white]| & |[fill=white]| & |[fill=white]| & |[fill=white]| & |[fill=black!60]| & |[fill=white]| \\
|[fill=white]| & |[fill=white]| & |[fill=white]| & |[fill=white]| & |[fill=white]| & |[fill=white]| & |[fill=white]| & |[fill=white]| & |[fill=white]| & |[fill=black!60]| & |[fill=white]| \\
|[fill=white]| & |[fill=white]| & |[fill=white]| & |[fill=white]| & |[fill=white]| & |[fill=white]| & |[fill=white]| & |[fill=white]| & |[fill=white]| & |[fill=white]| & |[fill=black!60]| \\
|[fill=white]| & |[fill=white]| & |[fill=white]| & |[fill=white]| & |[fill=white]| & |[fill=white]| & |[fill=white]| & |[fill=white]| & |[fill=white]| & |[fill=white]| & |[fill=black!60]| \\
|[fill=white]| & |[fill=white]| & |[fill=white]| & |[fill=white]| & |[fill=white]| & |[fill=white]| & |[fill=white]| & |[fill=white]| & |[fill=white]| & |[fill=white]| & |[fill=black!60]| \\
};
\end{tikzpicture}
\\
\downarrow
&
\downarrow
\\
\begin{tikzpicture}
\matrix[square matrix]{
|[fill=black!60]| & |[fill=black!60]| & |[fill=black!60]| & |[fill=white]| & |[fill=white]| & |[fill=white]| & |[fill=white]| & |[fill=white]| & |[fill=white]| & |[fill=white]| \\
|[fill=white]| & |[fill=white]| & |[fill=white]| & |[fill=black!60]| & |[fill=black!60]| & |[fill=black!60]| & |[fill=white]| & |[fill=white]| & |[fill=white]| & |[fill=white]| \\
|[fill=white]| & |[fill=white]| & |[fill=white]| & |[fill=white]| & |[fill=white]| & |[fill=white]| & |[fill=black!60]| & |[fill=black!60]| & |[fill=black!60]| & |[fill=white]| \\
|[fill=white]| & |[fill=white]| & |[fill=white]| & |[fill=white]| & |[fill=white]| & |[fill=white]| & |[fill=white]| & |[fill=white]| & |[fill=white]| & |[fill=black!60]| \\
|[fill=white]| & |[fill=white]| & |[fill=white]| & |[fill=white]| & |[fill=white]| & |[fill=white]| & |[fill=white]| & |[fill=white]| & |[fill=white]| & |[fill=black!60]| \\
|[fill=white]| & |[fill=white]| & |[fill=white]| & |[fill=white]| & |[fill=white]| & |[fill=white]| & |[fill=white]| & |[fill=white]| & |[fill=white]| & |[fill=black!60]| \\
|[fill=white]| & |[fill=white]| & |[fill=white]| & |[fill=white]| & |[fill=white]| & |[fill=white]| & |[fill=white]| & |[fill=white]| & |[fill=white]| & |[fill=black!60]| \\
};
\end{tikzpicture}
&
\begin{tikzpicture}
\matrix[square matrix]{
|[fill=black!60]| & |[fill=black!60]| & |[fill=black!60]| & |[fill=white]| & |[fill=white]| & |[fill=white]| & |[fill=white]| & |[fill=white]| & |[fill=white]| & |[fill=white]| & |[fill=white]| \\
|[fill=white]| & |[fill=white]| & |[fill=white]| & |[fill=black!60]| & |[fill=black!60]| & |[fill=black!60]| & |[fill=white]| & |[fill=white]| & |[fill=white]| & |[fill=white]| & |[fill=white]| \\
|[fill=white]| & |[fill=white]| & |[fill=white]| & |[fill=white]| & |[fill=white]| & |[fill=white]| & |[fill=black!60]| & |[fill=black!60]| & |[fill=black!60]| & |[fill=black!60]| & |[fill=white]| \\
|[fill=white]| & |[fill=white]| & |[fill=white]| & |[fill=white]| & |[fill=white]| & |[fill=white]| & |[fill=white]| & |[fill=white]| & |[fill=white]| & |[fill=white]| & |[fill=black!60]| \\
|[fill=white]| & |[fill=white]| & |[fill=white]| & |[fill=white]| & |[fill=white]| & |[fill=white]| & |[fill=white]| & |[fill=white]| & |[fill=white]| & |[fill=white]| & |[fill=black!60]| \\
|[fill=white]| & |[fill=white]| & |[fill=white]| & |[fill=white]| & |[fill=white]| & |[fill=white]| & |[fill=white]| & |[fill=white]| & |[fill=white]| & |[fill=white]| & |[fill=black!60]| \\
|[fill=white]| & |[fill=white]| & |[fill=white]| & |[fill=white]| & |[fill=white]| & |[fill=white]| & |[fill=white]| & |[fill=white]| & |[fill=white]| & |[fill=white]| & |[fill=black!60]| \\
};
\end{tikzpicture}
\\
\end{array}$
\end{center}
These switches all reduce the number of ones, contradicting that $M$ is a min-$2$TDS matrix.

Cases I--III prove the first half of the theorem statement.  Now let $a$ be the number of $J(3,1)$ components and let $b$ be the number of $J(1,3)$ components.

\textit{Case IV}: $M$ only has $J(1,3)$ and $J(3,1)$ components.  Then
\begin{align*}
n &= 3a+b, \\
m &= a+3b,
\end{align*}
and the number of ones in $M$ is $3(a+b)=3(n+m)/4=\lceil 3(n+m)/4 \rceil$.  In this case, we have $m \equiv 3n \pmod 8$, since by adding in a $J(1,3)$ or $J(3,1)$ component, we either increase $m$ by $1$ and $n$ by $3$, or we increase $n$ by $1$ and $m$ by $3$, and either way $m \equiv 3n \pmod 8$ remains true.

\textit{Case V}: $M$ has a $J(1,4)$ component and a $J(4,1)$ component.  Then
\begin{align*}
n &= 5+3a+b, \\
m &= 5+a+3b,
\end{align*}
and the number of ones in $M$ is $3(a+b)+8=3(n+m-10)/4+8=\lceil 3(n+m)/4 \rceil$.  In this case, we have $m \equiv 3n+6 \pmod 8$.

\textit{Case VI}: $M$ has a $J(1,5)$ component and a $J(4,1)$ component.  (Or, by symmetry, $M$ has a $J(1,4)$ component and a $J(5,1)$ component.)  Then
\begin{align*}
n &= 5+3a+b, \\
m &= 6+a+3b,
\end{align*}
and the number of ones in $M$ is $3(a+b)+9=3(n+m-11)/4+9=\lceil 3(n+m)/4 \rceil$.  In this case, we have $m \equiv 3n+7 \pmod 8$ (or $m \equiv 3n+3 \pmod 8$ in the transposed case).

\textit{Case VII}: $M$ has a $J(1,5)$ component and a $J(5,1)$ component.  Then
\begin{align*}
n &= 6+3a+b, \\
m &= 6+a+3b,
\end{align*}
and the number of ones in $M$ is $3(a+b)+10=3(n+m-12)/4+10=\lceil 3(n+m)/4 \rceil+1$.   In this case, we have $m \equiv 3n+4 \pmod 8$.

\textit{Case VIII}: $M$ has a $J(1,y)$ component with $4 \leq y \leq 7$ and no $J(x,1)$ component with $x \geq 4$.  (Or, by symmetry, $M$ has a $J(x,1)$ component with $4 \leq x \leq 7$ and no $J(1,y)$ component with $y \geq 4$.)  We have
\begin{align*}
n &= 1+3a+b \\
m &= y+a+3b
\end{align*}
and the number of ones in $M$ is
\begin{align*}
3(a+b)+y &= 3(n+m-1-y)/4+y \\
  &= 3(n+m)/4+(y-3)/4 \\
  &=
  \begin{cases}
  3(n+m)/4+1/4 & \text{ if $y=4$} \\
  3(n+m)/4+1/2 & \text{ if $y=5$} \\
  3(n+m)/4+3/4 & \text{ if $y=6$} \\
  3(n+m)/4+1 & \text{ if $y=7$}. \\
  \end{cases}
\end{align*}
Since the number of ones is an integer quantity, this is equal to $\lceil 3(n+m)/4 \rceil$ except when $y=7$ when it is equal to $\lceil 3(n+m)/4 \rceil+1$.  In this case we have $m \equiv 3n+y-3 \pmod 8$, i.e., $m \equiv 3n+1, 3n+2, 3n+3, 3n+4 \pmod 8$ when $y=4,5,6,7$, respectively (and $m \equiv 3n-3x+1 \pmod 8$ in the transposed case, i.e., $m \equiv 3n+5, 3n+2, 3n+7, 3n+4 \pmod 8$ when $x=4,5,6,7$, respectively).
\end{proof}

With our highly restricted families of $n \times m$ min-$2$TDS matrices, we find a general formula for $\gamma_{\times 2,t}(K_n \Box K_m)$ by simply counting the ones in all possible cases.  This gives the following proposition.  The subsequent Proposition~\ref{pr:k2realize} summarizes the matrices that need to be considered to find an example min-$2$TDS matrix for arbitrary $n \geq 1$ and $m \geq 1$, when possible.

\begin{prop}\label{pr:k2count}
For $n \geq 1$ and $m \geq n$, excluding $(n,m) \in \{(1,1),(1,2)\}$,
$$
\gamma_{\times 2,t}(K_n \Box K_m) =
\begin{cases}
3 & \text{if } n=1 \text{ and } m \geq 3, \\
2n & \text{if } n \geq 2 \text{ and } m \geq \lfloor (5n-4)/3 \rfloor+1, \\
\lceil 3(n+m)/4 \rceil+1 & \text{if } m \leq \lfloor (5n-4)/3 \rfloor \text{ and } m \equiv 3n+4 \pmod 8, \\
\lceil 3(n+m)/4 \rceil & \text{otherwise}. \\
\end{cases}
$$
Hence, in the square case, for $n \geq 2$,
$$\gamma_{\times 2,t}(K_n \Box K_n) =
\begin{cases}
\lceil 3n/2 \rceil+1 & \text{if } n \equiv 2 \pmod 4, \\
\lceil 3n/2 \rceil & \text{otherwise}. \\
\end{cases}
$$
\end{prop}

\begin{proof}
If $(n,m) \in \{(1,1),(1,2)\}$, then there are no $n \times m$ min-$2$TDS matrices.  If $n=1$ and $m \geq 3$, then any $1 \times m$ $(0,1)$-matrix with exactly $3$ ones is a min-$2$TDS matrix.  If $(n,m) \in \{(2,2),(3,3)\}$, then the following are $n \times m$ min-$2$TDS matrices:
\begin{center}
$\begin{array}{ccc}
\begin{tikzpicture}
\matrix[square matrix]{
|[fill=black!60]| & |[fill=black!60]| \\
|[fill=black!60]| & |[fill=black!60]| \\
};
\end{tikzpicture}
& \quad &
\begin{tikzpicture}
\matrix[square matrix]{
|[fill=black!60]| & |[fill=black!60]| & |[fill=black!60]| \\
|[fill=black!60]| & |[fill=white]| & |[fill=white]| \\
|[fill=black!60]| & |[fill=white]| & |[fill=white]| \\
};
\end{tikzpicture}
\\
\lceil 3(2+2)/4 \rceil+1=4 \text{ ones} & & \lceil 3(3+3)/4 \rceil=5 \text{ ones}
\end{array}$
\end{center}
Now assume $n \geq 2$ and $(n,m) \not\in \{(2,2),(3,3)\}$.

Let $M$ be an $n \times m$ min-$2$TDS matrix.  If $M$ has an all-$0$ row, then $M$ has at least $2n$ ones, as in the proof of Theorem~\ref{th:manycols}.  Likewise, if $M$ has an all-$0$ column, then $M$ has at least $2m \geq 2n$ ones.  If $M$ has no all-$0$ rows and no all-$0$ columns, then the number of ones in $M$ is
$$\begin{cases}
  \lceil 3(n+m)/4 \rceil +1 & \text{ if $m \equiv 3n+4 \pmod 8$} \\
  \lceil 3(n+m)/4 \rceil & \text{ otherwise}
  \end{cases}$$
by Theorem~\ref{th:compclass2}.  Hence
$$
\gamma_{\times 2,t}(K_n \Box K_m) =
\begin{cases}
3 & \text{if } n=1 \text{ and } m \geq 3, \\
\mathrm{min}(2n,\lceil 3(n+m)/4 \rceil+1) & \text{if } n \geq 2 \text{ and } m \equiv 3n+4 \pmod 8, \\
\mathrm{min}(2n,\lceil 3(n+m)/4 \rceil) & \text{otherwise}. \\
\end{cases}
$$
If $m>(5n-4)/3$ (which occurs when $m \geq \lfloor (5n-4)/3 \rfloor+1$), then $3(n+m)/4 > 2n-1$ implying $3(n+m)/4 \geq 2n$, in which case $$\mathrm{min}(2n,\lceil 3(n+m)/4 \rceil)=\mathrm{min}(2n,\lceil 3(n+m)/4 \rceil+1)=2n.$$  If $m \leq (5n-4)/3$, then $3(n+m)/4 \leq 2n-1<2n$, in which case $$\mathrm{min}(2n,\lceil 3(n+m)/4 \rceil)=\lceil 3(n+m)/4 \rceil$$ and
\begin{align*}
\mathrm{min}(2n,\lceil 3(n+m)/4 \rceil+1)=\lceil 3(n+m)/4 \rceil+1. & \qedhere
\end{align*}
\end{proof}

Theorem~\ref{th:totalprodbound} implies that if $n \geq 3$ and $m \geq 3$ then $$\gamma_{\times 2,t}(K_n \Box K_m) \geq \frac{1}{4} \gamma_{\times 2,t}(K_n) \gamma_{\times 2,t}(K_m)=\frac{9}{4}=2.25.$$	 In contrast, Proposition~\ref{pr:k2count} implies $$\gamma_{\times 2,t}(K_n \Box K_m)=\Theta(n)$$ when $3 \leq n \leq m$ and $n \rightarrow \infty$.

\begin{prop}\label{pr:k2realize}
Equality is realized in Proposition~\ref{pr:k2count} by the following min-$2$TDS matrices:
\begin{itemize}
 \item Two sporadic cases:
\begin{center}

\end{center}
 
 \item When $n=1$ and $m \geq 3$, any $1 \times m$ $(0,1)$-matrix with exactly $3$ ones, e.g.
 \begin{center}
%
\end{center}
 \item When $n \geq 2$ and $m \geq \lfloor (5n-4)/3 \rfloor+1$, any $n \times m$ $(0,1)$-matrix with two columns of ones and zeroes elsewhere, e.g.
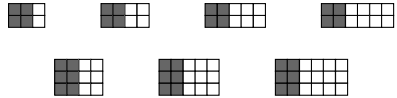
\begin{center}
%
\end{center}
\item When $n \geq 3$ and $m \leq \lfloor (5n-4)/3 \rfloor$, except $(n,m)=(3,3)$, the $(0,1)$-matrices with block structure
$$
\begin{array}{|c|c|c|}
\hline
A & \emptyset & \emptyset \\
\hline
\emptyset & B & \emptyset \\
\hline
\emptyset & \emptyset & C \\
\hline
\end{array}
$$
where $A=\emptyset$ or $A$ is one of the following:
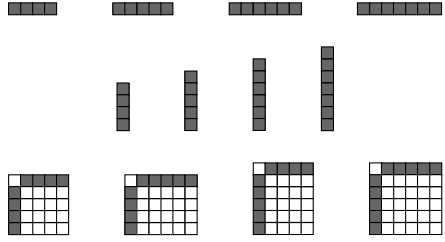
\begin{center}
%
\end{center}
and $B=\emptyset$ or one of the following:
\begin{center}
%
\end{center}
and so on, and $C=\emptyset$ or one of the following:
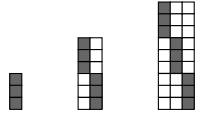
\begin{center}
%
\end{center}
and so on.
\end{itemize}
\end{prop}

\section{Concluding remarks}

Tables~\ref{ta:smallk2},~\ref{ta:smallk3} and~\ref{ta:smallk4} give examples of min-$k$TDS matrices for small $n$, $m$, and $k \in \{2,3,4\}$, found by computer search.  Where possible, we include a representative with no all-$0$ rows and no all-$0$ columns.  In some cases, the construction given in Proposition~\ref{pr:k2realize} (when $k=2$) is the only possible construction (up to permutations of the rows and columns, and matrix transposition).

\begin{table}[htp]
\centering
$\begin{array}{|c|c|c|c|c|c|c|c|}
\cline{2-8}
\multicolumn{1}{c|}{} & m=2 & m=3 & m=4 & m=5 & m=6 & m=7 & m=8 \\
\hline
n=2 &
\begin{tikzpicture}
\matrix[square matrix]{
|[fill=black!60]| & |[fill=black!60]| \\
|[fill=black!60]| & |[fill=black!60]| \\
};
\end{tikzpicture}
&
\begin{tikzpicture}
\matrix[square matrix]{
|[fill=white]| & |[fill=black!60]| & |[fill=black!60]| \\
|[fill=white]| & |[fill=black!60]| & |[fill=black!60]| \\
};
\end{tikzpicture}
&
\begin{tikzpicture}
\matrix[square matrix]{
|[fill=white]| & |[fill=white]| & |[fill=black!60]| & |[fill=black!60]| \\
|[fill=white]| & |[fill=white]| & |[fill=black!60]| & |[fill=black!60]| \\
};
\end{tikzpicture}
&
\begin{tikzpicture}
\matrix[square matrix]{
|[fill=white]| & |[fill=white]| & |[fill=white]| & |[fill=black!60]| & 
   |[fill=black!60]| \\
|[fill=white]| & |[fill=white]| & |[fill=white]| & |[fill=black!60]| & 
   |[fill=black!60]| \\
};
\end{tikzpicture}
&
\begin{tikzpicture}
\matrix[square matrix]{
|[fill=white]| & |[fill=white]| & |[fill=white]| & |[fill=white]| & 
   |[fill=black!60]| & |[fill=black!60]| \\
|[fill=white]| & |[fill=white]| & |[fill=white]| & |[fill=white]| & 
   |[fill=black!60]| & |[fill=black!60]| \\
};
\end{tikzpicture}
&
\begin{tikzpicture}
\matrix[square matrix]{
|[fill=white]| & |[fill=white]| & |[fill=white]| & |[fill=white]| & 
   |[fill=white]| & |[fill=black!60]| & |[fill=black!60]| \\
|[fill=white]| & |[fill=white]| & |[fill=white]| & |[fill=white]| & 
   |[fill=white]| & |[fill=black!60]| & |[fill=black!60]| \\
};
\end{tikzpicture}
&
\begin{tikzpicture}
\matrix[square matrix]{
|[fill=white]| & |[fill=white]| & |[fill=white]| & |[fill=white]| & 
   |[fill=white]| & |[fill=white]| & |[fill=black!60]| & |[fill=black!60]| \\
|[fill=white]| & |[fill=white]| & |[fill=white]| & |[fill=white]| & 
   |[fill=white]| & |[fill=white]| & |[fill=black!60]| & |[fill=black!60]| \\
};
\end{tikzpicture}
\\
\hline
\end{array}$
\caption{Small min-$2$TDS matrices.  Asterisks indicate that there are unlisted min-$2$TDS matrices with the same dimensions that are inequivalent under row and column permutations, and transposition (when $m=n$).}\label{ta:smallk2}
\end{table}

\begin{table}[htp]
\centering
$\begin{array}{|c|c|c|c|c|c|c|}
\cline{2-7}
\multicolumn{1}{c|}{} & m=3 & m=4 & m=5 & m=6 & m=7 & m=8 \\
\hline
n=2 & 
\begin{tikzpicture}
\matrix[square matrix]{
|[fill=black!60]| & |[fill=black!60]| & |[fill=black!60]| \\
|[fill=black!60]| & |[fill=black!60]| & |[fill=black!60]| \\
};
\end{tikzpicture}
&
\begin{tikzpicture}
\matrix[square matrix]{
|[fill=white]| & |[fill=black!60]| & |[fill=black!60]| & |[fill=black!60]| \\
|[fill=white]| & |[fill=black!60]| & |[fill=black!60]| & |[fill=black!60]| \\
};
\end{tikzpicture}
&
\begin{tikzpicture}
\matrix[square matrix]{
|[fill=white]| & |[fill=white]| & |[fill=black!60]| & |[fill=black!60]| & 
   |[fill=black!60]| \\
|[fill=white]| & |[fill=white]| & |[fill=black!60]| & |[fill=black!60]| & 
   |[fill=black!60]| \\
};
\end{tikzpicture}
&
\begin{tikzpicture}
\matrix[square matrix]{
|[fill=white]| & |[fill=white]| & |[fill=white]| & |[fill=black!60]| & 
   |[fill=black!60]| & |[fill=black!60]| \\
|[fill=white]| & |[fill=white]| & |[fill=white]| & |[fill=black!60]| & 
   |[fill=black!60]| & |[fill=black!60]| \\
};
\end{tikzpicture}
&
\begin{tikzpicture}
\matrix[square matrix]{
|[fill=white]| & |[fill=white]| & |[fill=white]| & |[fill=white]| & 
   |[fill=black!60]| & |[fill=black!60]| & |[fill=black!60]| \\
|[fill=white]| & |[fill=white]| & |[fill=white]| & |[fill=white]| & 
   |[fill=black!60]| & |[fill=black!60]| & |[fill=black!60]| \\
};
\end{tikzpicture}
&
\begin{tikzpicture}
\matrix[square matrix]{
|[fill=white]| & |[fill=white]| & |[fill=white]| & |[fill=white]| & 
   |[fill=white]| & |[fill=black!60]| & |[fill=black!60]| & |[fill=black!60]| \\
|[fill=white]| & |[fill=white]| & |[fill=white]| & |[fill=white]| & 
   |[fill=white]| & |[fill=black!60]| & |[fill=black!60]| & |[fill=black!60]| \\
};
\end{tikzpicture}
\\
\hline
\end{array}$
\caption{Small min-$3$TDS matrices.}\label{ta:smallk3}
\end{table}

\begin{table}[htp]
\centering
$\begin{array}{|c|c|c|c|c|c|}
\cline{2-6}
\multicolumn{1}{c|}{} & m=3 & m=4 & m=5 & m=6 & m=7 \\
\hline
n=2 &
-
&
\begin{tikzpicture}
\matrix[square matrix]{
|[fill=black!60]| & |[fill=black!60]| & |[fill=black!60]| & |[fill=black!60]| \\
|[fill=black!60]| & |[fill=black!60]| & |[fill=black!60]| & |[fill=black!60]| \\
};
\end{tikzpicture}
&
\begin{tikzpicture}
\matrix[square matrix]{
|[fill=white]| & |[fill=black!60]| & |[fill=black!60]| & |[fill=black!60]| & 
   |[fill=black!60]| \\
|[fill=white]| & |[fill=black!60]| & |[fill=black!60]| & |[fill=black!60]| & 
   |[fill=black!60]| \\
};
\end{tikzpicture}
&
\begin{tikzpicture}
\matrix[square matrix]{
|[fill=white]| & |[fill=white]| & |[fill=black!60]| & |[fill=black!60]| & 
   |[fill=black!60]| & |[fill=black!60]| \\
|[fill=white]| & |[fill=white]| & |[fill=black!60]| & |[fill=black!60]| & 
   |[fill=black!60]| & |[fill=black!60]| \\
};
\end{tikzpicture}
&
\begin{tikzpicture}
\matrix[square matrix]{
|[fill=white]| & |[fill=white]| & |[fill=white]| & |[fill=black!60]| & 
   |[fill=black!60]| & |[fill=black!60]| & |[fill=black!60]| \\
|[fill=white]| & |[fill=white]| & |[fill=white]| & |[fill=black!60]| & 
   |[fill=black!60]| & |[fill=black!60]| & |[fill=black!60]| \\
};
\end{tikzpicture}
\\
\hline
\end{array}$
\caption{Small min-$4$TDS matrices.}\label{ta:smallk4}
\end{table}

A natural way to extend this work is to find a general formula for $\gamma_{\times k,t}(K_n \Box K_m)$ in the $k=3$ or $k=4$ cases (as in Tables~\ref{ta:smallk3} and~\ref{ta:smallk4}).  It seems reasonable to suspect that the component switching method used here for $k=2$ will continue to be useful for larger $k$ values.  Other possible ways to extend this work are: (a) consider higher dimensions, e.g., $K_n \Box K_m \Box K_{\ell}$, and (b) consider graphs which have $K_n \Box K_m$ as a spanning subgraph, such as Latin rectangle graphs and the $n \times m$ queen's graph.


\end{document}